\documentclass[10pt, reqno]{amsart}
\usepackage{enumitem}
\usepackage{amsthm,amsmath,amssymb}
\usepackage{xcolor}
\usepackage{hyperref}
\hypersetup{colorlinks,citecolor=blue,linkcolor=blue,breaklinks=true}
\usepackage{mathrsfs}
\usepackage[textwidth=14.5cm]{geometry}
\usepackage{blindtext}
\usepackage{indentfirst}
\usepackage{braket}
\numberwithin{equation}{section}

\newtheorem{theorem}{Theorem}[section]

\theoremstyle{plain}
\newtheorem{proposition}[theorem]{Proposition}
\newtheorem{lemma}[theorem]{Lemma}
\newtheorem{corollary}[theorem]{Corollary}
\newtheorem{remark}[theorem]{Remark}

\def\be{\begin{equation}}
	\def\ee{\end{equation}}

\begin{document}
	\begin{sloppypar}

\title[]{New Volume Comparison Results and Volume Growth Rigidity of Gradient Ricci Almost Solitons }
\author{Li Wen-Qi}
\address{School of Mathematical Sciences , East China Normal University, Shanghai 200241, China}
\email{ 51255500054@stu.ecnu.edu.cn}

\begin{abstract}
	In this paper, we establish a new volume comparison theorem for complete Riemannian manifolds with Bakry--\'{E}mery Ricci curvature bounded below by a function $\rho(x)$. By constructing a twisted product  model space, we derive a new volume growth rigidity result for gradient Ricci almost solitons. A key feature of our approach is the characterization of the equality case through a generalized Riccati equation. Furthermore, we investigate shrinking gradient Ricci almost solitons by deriving potential function estimates, proving that the manifold is isometric to the Gaussian shrinking soliton if the volume growth bound is attained for all $r > 0$.
\end{abstract}

\maketitle
\date{}
\section{Introduction}

Let $(M^n,g)$ be a complete Riemannian manifold. We say that $M^n$ is a \emph{Ricci soliton} if there exists a smooth vector field $X$ on $M^n$ such that the Ricci curvature tensor $\mathrm{Ric}$ satisfies
\begin{equation}\label{sleq}
	\mathrm{Ric} + \frac{1}{2}\mathcal{L}_X g = \rho g,
\end{equation}
for some constant $\rho$, where $\mathcal{L}_X$ denotes the Lie derivative along $X$. In particular, if \(X=\nabla f\) for some \(f\in C^\infty(M^n)\), then \(f\)
is called the potential function and the soliton equation is expressed as
\[
    \operatorname{Ric}+\operatorname{Hess} f=\rho g,
\]
and \(M^n\) is referred to as a gradient Ricci soliton. According to the sign of \(\rho\), Ricci solitons are classified into three types: \emph{shrinking} ($\rho>0$), \emph{steady} ($\rho=0$), and \emph{expanding} ($\rho<0$). A typical example is the Gaussian shrinking Ricci soliton on $\mathbb{R}^n$, characterized by the potential function $f(x) = \frac{|x|^2}{4}$ and $\rho = \frac{1}{2}$. The tensor $\mathrm{Ric}_{f} := \mathrm{Ric} + \mathrm{Hess} f$ is known as the Bakry--\'{E}mery Ricci tensor, introduced by Bakry and \'{E}mery in \cite{BakryEmery1985}. This tensor plays a fundamental role in the field of geometric analysis.

Ricci solitons were first introduced by Hamilton \cite{Ha3} as self-similar solutions to the Ricci flow, and they play a fundamental role in Perelman's proof of the Poincar\'{e} conjecture \cite{Pel1,Pel2}. In recent years, gradient Ricci solitons have been extensively studied (see, e.g., \cite{CCZ,N,NW,Pw1,Pw2,Z} and references therein).

In \cite{PRR}, Pigola, Rigoli, Rimoldi, and Setti extended the notion of gradient Ricci solitons by allowing $\rho$ to be a smooth function, introducing the concept of \emph{gradient Ricci almost solitons}. The corresponding equation becomes
\[
\mathrm{Ric} + \mathrm{Hess} f = \rho(x) g.
\]
Similarly, it is called shrinking, steady, or expanding according as \(\rho(x)\) is
positive, zero, or negative on \(M^n\). Unlike the Einstein equation \(\operatorname{Ric}=\lambda g\), where Schur's
lemma forces \(\lambda\) to be constant in dimension \(n\ge3\), the presence of
the Hessian term in
\[
    \operatorname{Ric}+\nabla^2 f=\rho(x)g
\]
makes such a rigidity argument unavailable in general. This
represents a nontrivial generalization, as noted by Pigola--Rigoli--Rimoldi--Setti \cite{PRR}. A notable example of a nontrivial Ricci almost soliton (where $\rho$ is nonconstant) is the gradient Einstein soliton defined in \cite{CM}, which arises as a self-similar solution of the Ricci--Bourguignon flow \cite{CCD}. Therefore, the study of nontrivial Ricci almost solitons may contribute to understanding singularity formation in other geometric flows like the Ricci--Bourguignon flow.

Recently, significant research has been devoted to the study of gradient Ricci almost solitons (see, e.g., \cite{AE, ABR, FAJR, PRR}), with a particular emphasis on their scalar curvature rigidity. The primary objective of the present paper is to establish a new volume comparison theorem for complete Riemannian manifolds satisfying the following condition on the Bakry--\'{E}mery Ricci curvature:
\begin{equation}\label{RCG}
	\mathrm{Ric} + \nabla^2 f \ge \rho(x) g.
\end{equation}

Volume comparison theorems under the Bakry--Emery Ricci lower bound are well
understood when the lower bound is constant. For the weighted measure,
fundamental results were established by Wei--Wylie \cite{WW} and Bakry--Qian
\cite{BakryQian}. For the standard Riemannian measure, analogous estimates were
obtained by Zhang--Zhu \cite{ZZ} and Munteanu--Wang \cite{MW}. A natural
question is what happens when the constant lower bound is replaced by a
variable function \(\rho(x)\), a situation naturally arising for gradient Ricci
almost solitons.

The Bakry--Emery Ricci lower bound naturally yields a comparison inequality
for the weighted Laplacian of the distance function,
\[
    \Delta_f r
    =
    \Delta r-\langle\nabla f,\nabla r\rangle,
\]
rather than directly for the ordinary Laplacian \(\Delta r\). Therefore, in
order to obtain a comparison theorem for the standard Riemannian volume, one
must additionally control the radial derivative
\(\langle\nabla f,\nabla r\rangle\). When the lower bound is constant, Wei and
Wylie \cite{WW} treated this term under suitable assumptions and obtained local
volume comparison and rigidity results. When the lower bound is replaced by a
function \(\rho(x)\), however, both \(\rho\) and the radial behavior
of \(f\) vary along different geodesics, which makes the usual model-space
comparison more difficult.

In a recent paper, Azami and Hajiaghasi \cite{AzamiHajiaghasi2022}, following
an approach similar to that of Zhang and Zhu \cite{ZZ}, controlled the gradient
term by assuming
\[
    |\nabla f|(x)\le \frac{K}{d(x,o)^\alpha},
\]
where \(o\in M^n\) is fixed. Although this assumption makes the Bochner
argument applicable, it is not compatible with the standard noncompact
shrinking gradient Ricci solitons. Indeed, Cao and Zhou \cite{CZ} proved
quadratic growth estimates for the potential function of a complete noncompact
gradient shrinking Ricci soliton, and for the Gaussian shrinking soliton one
has \(f(x)=|x|^2/4\) and \(|\nabla f|(x)=|x|/2\). Thus the above decay condition
excludes the basic Gaussian shrinking soliton. Moreover, the resulting volume
comparison contains exponential correction factors arising from the estimate of
the gradient term, and such estimates are not well suited to identifying the
equality case or proving a sharp rigidity theorem.

Pigola, Rigoli, Rimoldi, and Setti \cite{PRR} also obtained volume comparison
results for Ricci almost solitons from the viewpoint of weighted manifolds.
Their estimates concern the weighted volume associated with the measure
\(e^{-f}d\operatorname{vol}\), under additional radial control on
\(\langle\nabla f,\nabla r\rangle\) or on the potential function \(f\). These
results are useful in the weighted setting, but they do not directly give a
sharp comparison theorem for the standard Riemannian volume or an equality
case for such a comparison. Thus, for the standard Riemannian volume under a
variable lower bound \(\rho(x)\), a different argument is needed.

To overcome these difficulties, we construct a function \(p(r,\theta)\)
satisfying a Riccati-type equation involving both the potential function \(f\)
and the variable lower bound \(\rho(x)\). This gives the model metric
\[
    g_p=dr^2+p(r,\theta)g_{\mathbb S^{n-1}}.
\]
Our method does not require any decay assumption on \(|\nabla f|\) and gives a
volume comparison theorem for the standard Riemannian volume. It also allows
us to describe the equality case and prove local and global rigidity results.
For shrinking gradient Ricci almost solitons, we further obtain quadratic
estimates for the potential function, volume growth of order at most \(R^n\)
under suitable assumptions on \(\rho\), and the conclusion that equality in the
sharp comparison for all radii forces the soliton to be the Gaussian shrinking
Ricci soliton.

Before stating our main results, we establish the necessary notation. For a fixed reference point $o \in M^n$, let $(r, \theta)$ denote the polar coordinates in a geodesic ball $B_o(R)$, where $r(x) = d(x, o)$ is the distance function and $\theta \in \mathbb{S}^{n-1}$. We denote by $\partial_r$ the radial vector field $\nabla r$. For any \(C^2\) function \(u\), its radial derivative is denoted by $\dot{u} = \partial_r u$, and we write $\ddot{u} = \mathrm{Hess} u (\partial_r, \partial_r)$ for the second radial derivative of $u$. Furthermore, let \(\omega(r,\theta)\) denote the mean curvature of the
geodesic sphere \(\partial B_o(r)\), and let \(J(r,\theta)d\theta\) denote the
Riemannian area element of \(\partial B_o(r)\) in geodesic polar coordinates.
Equivalently, the Riemannian volume element is written as
\[
    d\operatorname{vol}_g=J(r,\theta)\,dr\,d\theta.
\]
We denote by \(\operatorname{Vol}(B_o(r))\) the Riemannian volume of the
geodesic ball of radius \(r\), and by \(\operatorname{inj}(o)\) the injectivity
radius at \(o\).

Our first result controls the volume growth of geodesic balls in a complete
manifold satisfying \eqref{RCG} by using a chosen test function \(p\). Here the
test function \(p\) determines the radial part of the model metric used for
comparison.

\begin{theorem}\label{thm1.1}
Let \((M^n,g)\) be a complete Riemannian manifold satisfying \eqref{RCG} for some \(f,\rho\in C^2(M)\). Fix a point \(o\in M\), and let
\(\operatorname{inj}(o)>R_0>0\). Let \(p\in C^2(B_o(R_0)\setminus\{o\})\) be a positive function. In geodesic
polar coordinates centered at \(o\), we write
\[
    p(x)=p(r,\theta).
\] Assume that \(p\) admits a continuous extension to \(o\) with \(p(o)=0\), and
that, as \(r\to0^+\),
\[
    p(r,\theta)=r^2+O(r^4),\qquad
    \dot p(r,\theta)=2r+O(r^3),\qquad
    \ddot p(r,\theta)=2+O(r^2),
\]
uniformly in \(\theta\in\mathbb S^{n-1}\). Define
\[
    \mathcal K_p(r,\theta)
    =
    \frac{n-1}{2}\frac{\ddot p(r,\theta)}{p(r,\theta)}
    -
    \frac{n-1}{4}\frac{\dot p^2(r,\theta)}{p^2(r,\theta)}.
\]

Then, for every \(0<R\le R_0\), the volume of \(B_o(R)\) satisfies
\begin{equation}\label{mvcp}
\begin{aligned}
    \operatorname{Vol}(B_o(R))
    \le
    \int_{\mathbb S^{n-1}}
    \int_0^R
    p^{\frac{n-1}{2}}(s,\theta)
    \exp\left\{
        \int_0^s
        \frac{1}{p(t,\theta)}
        \int_0^t
        p(u,\theta)
        \bigl(
            \ddot f(u,\theta)
            -
            \rho(u,\theta)
            -
            \mathcal K_p(u,\theta)
        \bigr)
        \,du\,dt
    \right\}
    ds\,d\theta .
\end{aligned}
\end{equation}

Moreover, equality holds in \eqref{mvcp} for \(B_o(R)\) if and only if, for all
\(0<r\le R\) and all \(\theta\in\mathbb S^{n-1}\), the chosen function \(p\)
satisfies
\begin{equation}\label{ricatii}
    \ddot f(r,\theta)-\rho(r,\theta)
    =
    \mathcal K_p(r,\theta),
\end{equation}
and
\begin{equation}\label{eqst}
    \omega(r,\theta)
    =
    \frac{(n-1)\dot p(r,\theta)}{2p(r,\theta)}.
\end{equation}

In the equality case,
\[
    J(r,\theta)=p^{\frac{n-1}{2}}(r,\theta)
\]
for all \(0<r\le R\), and hence
\[
    \operatorname{Vol}(B_o(r))
    =
    \int_{\mathbb S^{n-1}}
    \int_0^r
    p^{\frac{n-1}{2}}(s,\theta)\,ds\,d\theta
    \qquad
    \text{for all }0<r\le R.
\]
\end{theorem}
\begin{remark}
	Let $q(x) = \sqrt{p(x)}$. Then equation \eqref{ricatii} becomes
	\[
	\frac{\ddot{f}(r,\theta) - \rho(r,\theta)}{n-1} = \dot{h}(r,\theta) + h^2(r,\theta),
	\]
	where $h = \frac{\dot{q}}{q}$. Hence, \eqref{ricatii} is a Riccati equation.
\end{remark}

\begin{remark}\label{rem:well-defined-exponential}
The asymptotic assumptions on \(p\) ensure that the exponential factor in
\eqref{mvcp} is well defined near the pole. Indeed,
\[
    p(r,\theta)=r^2+O(r^4),\qquad
    \dot p(r,\theta)=2r+O(r^3),\qquad
    \ddot p(r,\theta)=2+O(r^2)
\]
imply
\[
    \frac{\ddot p(r,\theta)}{p(r,\theta)}
    =
    \frac{2}{r^2}+O(1),
    \qquad
    \frac{\dot p^2(r,\theta)}{p^2(r,\theta)}
    =
    \frac{4}{r^2}+O(1).
\]
Hence
\[
    \mathcal K_p(r,\theta)=O(1)
    \qquad
    \text{as }r\to0^+,
\]
uniformly in \(\theta\). Since \(f,\rho\in C^2(M)\), the quantity
\(\ddot f-\rho-\mathcal K_p\) is locally bounded near \(o\). Therefore
\[
    p(u,\theta)\bigl(\ddot f(u,\theta)-\rho(u,\theta)
    -\mathcal K_p(u,\theta)\bigr)=O(u^2),
\]
and consequently
\[
    \frac{1}{p(t,\theta)}
    \int_0^t
    p(u,\theta)
    \bigl(\ddot f(u,\theta)-\rho(u,\theta)-\mathcal K_p(u,\theta)\bigr)
    \,du
    =
    O(t).
\]
Thus the outer integral in the exponent of \eqref{mvcp} is finite near
\(t=0\).
\end{remark}
\begin{remark}
	The condition \(p(s,\theta)=s^2+O(s^4)\) is imposed to ensure the correct
behavior near the origin. For any smooth Riemannian metric, the mean curvature of small geodesic spheres
satisfies
\[
    \omega(s,\theta)=\frac{n-1}{s}+O(s),
\]
as \(s\to0^+\). This follows from the standard expansion of the volume density
in geodesic polar coordinates; see, for example, Gray--Vanhecke \cite{GrayVanhecke1979}. 

	If $p$ had an $s^3$ term, say $p = s^2 + as^3 + O(s^4)$, then the model mean curvature would become
	\[
	m_p = \frac{n-1}{2}\frac{\dot{p}}{p} = \frac{n-1}{s} + \frac{n-1}{2}a + O(s).
	\]
	Hence $a$ must be zero, which gives $\omega - m_p = O(s)$ as $s \to 0$. This ensures that the integrand in the exponential term of \eqref{mvcp} stays bounded near the origin, and that $J/p^{(n-1)/2} \to 1$ with zero radial derivative at $o$.
\end{remark}
This result describes the behavior of the volume of geodesic balls. When equality holds, the test function \(p(r,\theta)\) satisfies the Riccati
equation \eqref{ricatii} in every radial direction. In this case, the volume growth of $M^n$ is controlled by the volume growth of a geodesic ball in a Riemannian manifold with metric $g = dr^2 + p(r, \theta)  g_{\mathbb{S}^{n-1}}$.

In the classical Bishop--Gromov volume comparison theorem, the equality case
implies that the geodesic ball is isometric to a ball in the corresponding
model space. More recently, Cheng, Ribeiro, and Zhou \cite{CRZ} proved that
gradient shrinking Ricci solitons with the same volume growth as Euclidean
space satisfy a similar rigidity property. Their argument uses the scalar
curvature equation of a gradient shrinking Ricci soliton and the strong maximum
principle. It is therefore natural to ask whether such a rigidity conclusion
can be obtained directly from the equality case of a volume comparison theorem.

The equality characterization in Theorem~\ref{thm1.1} leads to the following
local rigidity result. Thus, the next theorem gives an analogue of the classical
Bishop--Gromov equality statement under the Bakry--Emery lower bound
\eqref{RCG}.

\begin{theorem}\label{thm3.1}
Let \((M^n,g)\) be a complete Riemannian manifold satisfying \eqref{RCG} for
some \(f,\rho\in C^2(M)\). Fix a point \(o\in M\). Suppose that, for some
\(R>0\), there exists a positive function $p\in C^2(B_o(R)\setminus\{o\})$ satisfying equation \eqref{ricatii} on \(B_o(R)\setminus\{o\}\). Assume that
\(p\) admits a continuous extension to \(o\) with \(p(o)=0\), and that, as
\(r\to0^+\),
\[
    p(r,\theta)=r^2+O(r^4),\qquad
    \dot p(r,\theta)=2r+O(r^3),\qquad
    \ddot p(r,\theta)=2+O(r^2),
\]
uniformly in \(\theta\in\mathbb S^{n-1}\). Then for every
\(0<r\le \min\{R,\operatorname{inj}(o)\}\), we have
\begin{equation}\label{vbcp}
    \operatorname{Vol}(B_o(r))
    \le
    \int_{\mathbb S^{n-1}}\int_0^r
    p^{(n-1)/2}(s,\theta)\,ds\,d\theta .
\end{equation}
Moreover, equality holds in the above inequality for some
\(0<r_0\le \min\{R,\operatorname{inj}(o)\}\) if and only if, in geodesic polar
coordinates on \(B_o(r_0)\setminus\{o\}\), the metric is given by
\[
    g=dr^2+p(r,\theta)g_{\mathbb S^{n-1}}.
\]
In particular, \(B_o(r_0)\) is isometric to the corresponding polar model ball.
\end{theorem}

\begin{remark}
In Theorem \ref{thm3.1}, the function \(p\) should be understood first as a
test function which determines the comparison term \(p^{(n-1)/2}\). Thus the
expression
\[
    g_p=dr^2+p(r,\theta)g_{\mathbb S^{n-1}}
\]
should not, in general, be regarded as a smooth model metric at the pole. For
an arbitrary test function \(p\), this expression may fail to extend smoothly
across \(r=0\).

There are standard choices of \(p\) for which no such difficulty occurs. For
example, in the radial case, if
\[
    p(r)=\psi^2(r),
    \qquad
    \psi(r)=r\varphi(r^2),
\]
where \(\varphi\) is smooth near \(0\) and \(\varphi(0)=1\), then
\[
    dr^2+p(r)g_{\mathbb S^{n-1}}
\]
extends smoothly across \(r=0\). This includes the usual space-form examples
\(p(r)=\operatorname{sn}_K^2(r)\).

In the equality case of Theorem \ref{thm3.1}, no additional smoothness
assumption on \(g_p\) is needed. Indeed, equality implies that the original
smooth metric \(g\), written in geodesic polar coordinates on the geodesic
ball, has the form
\[
    g=dr^2+p(r,\theta)g_{\mathbb S^{n-1}}.
\]
Thus the required smooth behavior at the pole follows from the smoothness of
the original metric \(g\).
\end{remark}
\begin{remark}
	In Theorem \ref{thm3.1}, taking $\rho = (n-1)C$ and $f = 0$, and choosing $p(x) = \mathrm{sn}^2_C(r)$, we recover the volume comparison theorem for manifolds with lower Ricci curvature bound. Here, $\mathrm{sn}_C(r)$ is the solution of
	\begin{equation}\label{vineq}
		\begin{cases}
			\ddot{\mathrm{sn}}_C(r) + C \mathrm{sn}_C(r) = 0, \\
			\mathrm{sn}_C(0) = 0, \quad \dot{\mathrm{sn}}_C(0) = 1.
		\end{cases}
	\end{equation}
\end{remark}

The volume comparison theorem provides a powerful tool for establishing local rigidity. To obtain global rigidity, however, one has to control the cut locus.
Assuming equality in Theorem~\ref{thm3.1} holds for a test function
\(p(r,\theta)\) that is strictly positive for all \(r>0\), the corresponding
polar model is globally defined in the radial variable. The global volume
equality then forces the cut locus of \(o\) in \(M\) to be empty. Consequently, the volume comparison theorem extends to the entire manifold and yields the following global rigidity theorem.
\begin{theorem}\label{hhg}
Let \((M^n,g)\) be a complete noncompact Riemannian manifold satisfying
\eqref{RCG} for some smooth functions \(f\) and \(\rho\). Fix \(o\in M\).
Let \(p=p(r,\theta)\) be a positive smooth function on
\((0,\infty)\times\mathbb S^{n-1}\), satisfying the asymptotic condition in
Theorem \ref{thm1.1}, and suppose that \(p\) solves \eqref{ricatii} for all \(r>0\) and all
\(\theta\in\mathbb S^{n-1}\). Assume that the metric
\[
    g_p=dr^2+p(r,\theta)g_{\mathbb S^{n-1}}
\]
extends smoothly across \(r=0\) and is globally defined for all \(r>0\).
If, for every \(R>0\),
\begin{equation}\label{eq:global-volume-equality}
     \operatorname{Vol}(B_o(R))
    =
    \int_{\mathbb S^{n-1}}\int_0^R
    p^{\frac{n-1}{2}}(s,\theta)\,ds\,d\theta,
\end{equation}
then the cut locus of $o$ is empty. Hence
\(\exp_o:T_oM\to M\) gives a global diffeomorphism between \(M^n\) and
\(\mathbb R^n\), and
\[
    \pi_k(M^n)=0
    \qquad
    \text{for all }k\ge1.
\]
Moreover, \((M^n,g)\) is globally isometric to the smooth model manifold
whose metric is written in global polar coordinates as
\[
    g_p=dr^2+p(r,\theta)g_{\mathbb S^{n-1}}.
\]
\end{theorem}
 
We now apply the comparison theorem to volume growth estimates for gradient
shrinking Ricci almost solitons. Since a normalized shrinking almost soliton
with \(\rho\ge \frac12\) satisfies the Bakry--Emery lower bound
\[
    \operatorname{Ric}+\nabla^2 f\ge \frac12 g,
\]
we first recall two related results under this curvature condition.

Munteanu and Wang \cite{MW} proved that if a complete smooth metric measure
space satisfies
\[
    \operatorname{Ric}+\nabla^2 f\ge \frac12 g,
    \qquad
    |\nabla f|^2\le f,
\]
then
\[
    \operatorname{Vol}(B_p(r))\le c(n)e^{f(p)}r^n
\]
for all \(r>0\). Cheng--Ribeiro--Zhou \cite{CRZ} later obtained a sharp volume
growth estimate for complete noncompact normalized gradient shrinking Ricci
solitons. In particular, their estimate for geodesic balls is achieved by the
Gaussian shrinking Ricci soliton, and the equality case characterizes the
Gaussian shrinking Ricci soliton.

The estimate below keeps track of the possible failure of the condition
\(|\nabla f|^2\le f\) through the quantity \(A_f(R)\). Moreover, equality for
all radii characterizes the Gaussian shrinking Ricci soliton.

In the normalized shrinking case considered below, the reference lower bound is
\(\frac12\). For a fixed point \(o\in M\), define
\[
    \Phi(r,\theta)
    :=
    f(o)
    -
    \frac1r
    \int_0^r
    \bigl(f-|\nabla f|^2\bigr)(t,\theta)\,dt .
\]
For every \(R>0\), set
\[
    A_f(R)
    :=
    \sup_{B_o(R)}
    \bigl(|\nabla f|^2-f\bigr).
\]
Since \(B_o(R)\) is compact, \(A_f(R)<\infty\) for every fixed \(R>0\).
\begin{theorem}\label{thm1.3}
Let \((M^n,g,f,\rho)\), \(n\ge2\), be a complete noncompact gradient shrinking
Ricci almost soliton
\[
    \operatorname{Ric}+\nabla^2f=\rho g .
\]
Assume that
\[
    \rho(x)\ge \frac12
    \qquad
    \text{for all }x\in M.
\]
Then, for every \(R>0\), the geodesic ball \(B_o(R)\) satisfies
\begin{equation}\label{eq:sharp-comparison-almost}
    \operatorname{Vol}(B_o(R))
    \le
    \int_{\mathbb S^{n-1}}
    \int_0^R
    e^{\Phi(s,\theta)}s^{n-1}\,ds\,d\theta .
\end{equation}
Consequently,
\begin{equation}\label{eq:polynomial-growth-almost}
    \operatorname{Vol}(B_o(R))
    \le
    e^{f(o)+A_f(R)}\omega_nR^n,
\end{equation}
where \(\omega_n\) denotes the volume of the unit ball in \(\mathbb R^n\).

Moreover, if for some \(R>0\),
\[
    f(o)\le -A_f(R),
\]
then
\begin{equation}\label{eq:euclidean-volume-bound}
    \operatorname{Vol}(B_o(R))\le \omega_nR^n .
\end{equation}

Finally, if equality in \eqref{eq:sharp-comparison-almost} holds for every
\(R>0\), then \((M^n,g,f,\rho)\) is the Gaussian shrinking Ricci soliton; that
is,
\[
    (M^n,g)\cong(\mathbb R^n,g_{\mathrm{Euc}}),
    \qquad
    \rho\equiv\frac12,
    \qquad
    f(x)=f(o)+\frac{d^2(o,x)}4.
\]
\end{theorem}
\begin{remark}
When \(\rho\equiv\frac12\) and the potential is normalized by
\[
    R+|\nabla f|^2=f,
\]
we have
\[
    f-|\nabla f|^2=R.
\]
Therefore Theorem \ref{thm1.3} recovers the volume comparison theorem of
Cheng--Ribeiro--Zhou \cite{CRZ}:
\[
\operatorname{Vol}(B_o(R))
\le
\int_{\mathbb S^{n-1}}
\int_0^R
\exp\left\{
    f(o)-\frac1s\int_0^sR(t,\theta)\,dt
\right\}
s^{n-1}\,ds\,d\theta .
\]
The equality case is obtained here from the equality case of Theorem
\ref{thm1.1} and the global rigidity Theorem \ref{hhg}, rather than from
scalar curvature rigidity.
\end{remark}

\begin{remark}
The estimate in Theorem \ref{thm1.3} does not require the exact almost soliton
equation
\[
    \operatorname{Ric}+\nabla^2 f=\rho g
\]
itself. The same estimate remains valid under the weaker Bakry--Emery lower
bound
\[
    \operatorname{Ric}+\nabla^2 f\ge \frac12 g.
\]
Under this weaker assumption, the same equality rigidity also holds: if equality
in \eqref{eq:sharp-comparison-almost} holds for all \(R>0\), then \((M^n,g)\)
is globally isometric to \(\mathbb R^n\), and
\[
    f(x)=f(o)+\frac{d^2(o,x)}4.
\]

In the special case of a normalized gradient shrinking Ricci soliton, where
\[
    \rho\equiv\frac12,
    \qquad
    R+|\nabla f|^2=f,
\]
the estimate recovers the Cheng--Ribeiro--Zhou \cite{CRZ} volume growth estimate for
geodesic balls, and the equality case gives the Gaussian shrinking soliton.
\end{remark}
The estimate \eqref{eq:polynomial-growth-almost} still involves the
radius-dependent quantity \(A_f(R)\). We later control this quantity by means
of quadratic estimates for the potential function, and thereby obtain volume
growth of order at most \(R^n\).

\section{New volume comparison results}
\label{sec:volume_comparison}

In this section, we present the proof of Theorem \ref{thm1.1} and derive related volume comparison theorems for geodesic spheres and balls.

\begin{proof}[Proof of Theorem \ref{thm1.1}]
    We begin by establishing the geometric quantities along radial geodesics. Set
    \[
        a := \frac{n-1}{2}.
    \]

    Let \(J(r,\theta)d\theta\) be the Riemannian area element of the geodesic
sphere in geodesic polar coordinates, and let
\[
    \omega(r,\theta)=\partial_r\log J(r,\theta)
\]
be its mean curvature. Denote by
\[
    \mathcal{A}=\operatorname{Hess} r\big|_{\{\partial_r\}^{\perp}}
\]
the second fundamental form of the geodesic sphere, which gives
\[
    \omega=\operatorname{tr}\mathcal{A}.
\]

    The standard Riccati identity for the distance function is given by
    \[
        \dot{\omega}+|\mathcal{A}|^2+\operatorname{Ric}(\partial_r,\partial_r)=0.
    \]
    By introducing the trace-free part of the second fundamental form, we can
    express this identically as
    \begin{equation}\label{eq:riccati-with-defect}
        \dot{\omega}+\frac{\omega^2}{n-1}
        +\operatorname{Ric}(\partial_r,\partial_r)=-D_0,
    \end{equation}
    where \(D_0\) is defined by
    \[
        D_0:=|\mathcal{A}|^2-\frac{\omega^2}{n-1}\ge0,
    \]
    with the non-negativity guaranteed by the Cauchy-Schwarz inequality.

    Multiplying \eqref{eq:riccati-with-defect} by the test function
    \(p(t,\theta)\) and integrating from \(\varepsilon\) to \(r\), we obtain
    \[
        \int_\varepsilon^r p(t,\theta)\dot{\omega}(t,\theta)\,dt
        +\int_\varepsilon^r
        \frac{p(t,\theta)\omega^2(t,\theta)}{n-1}\,dt
        +\int_\varepsilon^r
        p(t,\theta)\operatorname{Ric}(\partial_r,\partial_r)(t,\theta)\,dt
        =
        -\int_\varepsilon^r p(t,\theta)D_0(t,\theta)\,dt.
    \]
    Integrating the first term by parts and completing the square for
    \(\omega\) yields
    \begin{align}
        & p(r,\theta)\omega(r,\theta)
        -p(\varepsilon,\theta)\omega(\varepsilon,\theta)
        +\int_\varepsilon^r
        \frac{p(t,\theta)}{n-1}
        \left(
            \omega(t,\theta)
            -\frac{(n-1)\dot{p}(t,\theta)}{2p(t,\theta)}
        \right)^2dt
        +\int_\varepsilon^r p(t,\theta)D_0(t,\theta)\,dt
        \notag \\
        &\qquad =
        \int_\varepsilon^r
        \left(
            \frac{(n-1)\dot{p}^2(t,\theta)}{4p(t,\theta)}
            -
            p(t,\theta)\operatorname{Ric}(\partial_r,\partial_r)(t,\theta)
        \right)dt .
        \label{eq:complete-square}
    \end{align}

    By the asymptotic assumption on \(p\) in Theorem \ref{thm1.1}, we have
\[
    p(\varepsilon,\theta)=\varepsilon^2+O(\varepsilon^4).
\]
Together with
\[
    \omega(\varepsilon,\theta)=\frac{n-1}{\varepsilon}+O(\varepsilon),
\]
this gives
\[
    p(\varepsilon,\theta)\omega(\varepsilon,\theta)=O(\varepsilon)\to0
    \qquad\text{as }\varepsilon\to0^+.
\]
Letting \(\varepsilon\to0\) in
    \eqref{eq:complete-square} and utilizing \eqref{RCG} we have
    \[
        \operatorname{Ric}(\partial_r,\partial_r)
        +\ddot{f}-\rho
        =
        \bigl(\operatorname{Ric}+\operatorname{Hess}f-\rho g\bigr)
        (\partial_r,\partial_r)
        \ge0,
    \]
    we deduce the exact identity
    \begin{equation}\label{eq:key-identity}
        p(r,\theta)\omega(r,\theta)+E(r,\theta)
        =
        \int_0^r
        \left(
            p(t,\theta)(\ddot{f}-\rho)(t,\theta)
            +
            \frac{(n-1)\dot{p}^2(t,\theta)}{4p(t,\theta)}
        \right)dt,
    \end{equation}
    where
    \begin{equation}\label{eq:defect-E}
        E(r,\theta)
        :=
        \int_0^r
        \bigg[
            \frac{p(t,\theta)}{n-1}
            \left(
                \omega(t,\theta)
                -
                \frac{(n-1)\dot{p}(t,\theta)}{2p(t,\theta)}
            \right)^2
            +
            p(t,\theta)D_0(t,\theta)
            +
            p(t,\theta)
            \bigl(
                \operatorname{Ric}(\partial_r,\partial_r)(t,\theta)
                +\ddot{f}(t,\theta)
                -\rho(t,\theta)
            \bigr)
        \bigg]dt
        \ge0.
    \end{equation}
The integral defining \(E(r,\theta)\) is well defined near \(t=0\). Indeed,
the standard expansions in geodesic polar coordinates give
\[
    \omega(t,\theta)=\frac{n-1}{t}+O(t),
    \qquad
    \mathcal A(t,\theta)=\frac1t I+O(t),
\]
as \(t\to0^+\). On the other hand, the assumptions on \(p\) imply
\[
    \frac{(n-1)\dot p(t,\theta)}{2p(t,\theta)}
    =
    \frac{n-1}{t}+O(t).
\]
Hence
\[
    \omega(t,\theta)
    -
    \frac{(n-1)\dot p(t,\theta)}{2p(t,\theta)}
    =
    O(t),
    \qquad
    D_0(t,\theta)=O(t^2).
\]
Since \(p(t,\theta)=t^2+O(t^4)\), the first two terms in the integrand of
\eqref{eq:defect-E} are \(O(t^4)\). Moreover,
\(\operatorname{Ric}(\partial_r,\partial_r)\), \(\ddot f\), and \(\rho\) are
bounded near \(o\), and therefore the last term is \(O(t^2)\). Thus the
integrand in \eqref{eq:defect-E} is locally integrable near \(t=0\), and
\(E(r,\theta)\) is well defined.

    Dividing \eqref{eq:key-identity} by \(p(r,\theta)\) and subtracting the exact
    logarithmic derivative
    \[
        \partial_r\log \bigl(p(r,\theta)^a\bigr)
        =
        a\frac{\dot{p}(r,\theta)}{p(r,\theta)}
        =
        \frac{n-1}{2}\frac{\dot{p}(r,\theta)}{p(r,\theta)},
    \]
    we arrive at the differential equation for the Jacobian ratio:
    \begin{equation}\label{eq:ratio-before-K}
        \partial_r\log
        \left(
            \frac{J(r,\theta)}{p(r,\theta)^a}
        \right)
        =
        \frac{1}{p(r,\theta)}
        \int_0^r
        \left(
            p(t,\theta)(\ddot{f}-\rho)(t,\theta)
            +
            \frac{(n-1)\dot{p}^2(t,\theta)}{4p(t,\theta)}
        \right)dt
        -
        \frac{n-1}{2}\frac{\dot{p}(r,\theta)}{p(r,\theta)}
        -
        \frac{E(r,\theta)}{p(r,\theta)}.
    \end{equation}
By the asymptotic assumption on \(p\), we have \(\dot p(0,\theta)=0\). Hence
\[
    \frac{n-1}{2}\dot p(r,\theta)
    =
    \int_0^r \frac{n-1}{2}\ddot p(t,\theta)\,dt.
\] 
    Substituting this into \eqref{eq:ratio-before-K} and factoring out
    \(p(t,\theta)\) from the integrand refines the equation to
    \begin{equation}\label{eq:ratio-differential}
        \partial_r\log
        \left(
            \frac{J(r,\theta)}{p(r,\theta)^a}
        \right)
        =
        \frac{1}{p(r,\theta)}
        \int_0^r
        p(t,\theta)
        \bigl(
            \ddot{f}(t,\theta)
            -
            \rho(t,\theta)
            -
            \mathcal{K}_p(t,\theta)
        \bigr)dt
        -
        \frac{E(r,\theta)}{p(r,\theta)},
    \end{equation}
    where
    \begin{equation}\label{eq:correct-Kp}
        \mathcal{K}_p(r,\theta)
        =
        \frac{n-1}{2}
        \frac{\ddot{p}(r,\theta)}{p(r,\theta)}
        -
        \frac{n-1}{4}
        \frac{\dot{p}^2(r,\theta)}{p^2(r,\theta)}.
    \end{equation}
    Since \(E(r,\theta)\ge0\), we obtain the differential inequality
    \begin{equation}\label{eq:ratio-ineq}
        \partial_r\log
        \left(
            \frac{J(r,\theta)}{p(r,\theta)^a}
        \right)
        \le
        \frac{1}{p(r,\theta)}
        \int_0^r
        p(t,\theta)
        \bigl(
            \ddot{f}(t,\theta)
            -
            \rho(t,\theta)
            -
            \mathcal{K}_p(t,\theta)
        \bigr)dt.
    \end{equation}

    By the standard expansion in normal coordinates, we have
    \(J(r,\theta)=r^{n-1}(1+O(r^2))\) and
    \(p(r,\theta)^a=r^{n-1}(1+O(r^2))\). Thus,
    \[
        \frac{J(r,\theta)}{p(r,\theta)^a}\to1
    \]
    as \(r\to0^+\). Integrating \eqref{eq:ratio-ineq} from \(0\) to
    \(r<R\le \operatorname{inj}(o)\) establishes the pointwise area-density
    bound:
    \begin{equation}\label{eq:J-comparison}
        J(r,\theta)
        \le
        p(r,\theta)^a
        \exp
        \left\{
            \int_0^r
            \frac{1}{p(s,\theta)}
            \int_0^s
            p(t,\theta)
            \bigl(
                \ddot{f}(t,\theta)
                -
                \rho(t,\theta)
                -
                \mathcal{K}_p(t,\theta)
            \bigr)dt\,ds
        \right\}.
    \end{equation}

    Let us define the right-hand side of \eqref{eq:J-comparison} as
    \(\Psi(r,\theta)\). Then the inequality reads
    \begin{equation}\label{J controll}
       J(r,\theta)\le \Psi(r,\theta)
       \qquad \text{for }0<r<R.
    \end{equation}

    Substituting the definition of \(\Psi(r,\theta)\), we obtain the desired
    volume comparison theorem:
    \begin{equation}\label{eq:correct-volume-comparison}
        \operatorname{Vol}(B_o(R))
        \le
        \int_{\mathbb{S}^{n-1}}
        \int_0^R
        p(s,\theta)^{\frac{n-1}{2}}
        \exp
        \left\{
            \int_0^s
            \frac{1}{p(t,\theta)}
            \int_0^t
            p(u,\theta)
            \bigl(
                \ddot{f}(u,\theta)
                -
                \rho(u,\theta)
                -
                \mathcal{K}_p(u,\theta)
            \bigr)du\,dt
        \right\}ds\,d\theta.
    \end{equation}

    We now turn to the equality case. Suppose equality holds in
    \eqref{eq:correct-volume-comparison} on \(B_o(R)\). Since \(J\le\Psi\) and
    both functions are continuous on \((0,R]\times\mathbb{S}^{n-1}\), the
    equality of their integrals implies that the integrands must coincide
    pointwise:
    \[
        J(s,\theta)=\Psi(s,\theta)
        \qquad
        \text{for all }0<s\le R,\ \theta\in\mathbb{S}^{n-1}.
    \]
    Comparing this exact identity with the differential formula
    \eqref{eq:ratio-differential}, we immediately deduce that
    \[
        \int_0^r
        \frac{E(s,\theta)}{p(s,\theta)}\,ds=0
        \qquad
        \text{for every }0<r\le R.
    \]
    Since \(E(s,\theta)\ge0\) and \(p(s,\theta)>0\) for \(s>0\), it follows that
    the cumulative error vanishes:
    \[
        E(s,\theta)=0
        \qquad
        \text{for all }0<s\le R.
    \]
    By the definition of \(E(r,\theta)\) in \eqref{eq:defect-E}, each
    non-negative term in the integrand must vanish identically. This
    simultaneously forces three pointwise conditions. First, the completed
    square term vanishes:
    \begin{equation}\label{eq:equality-mean-curvature}
        \omega(r,\theta)
        =
        \frac{(n-1)\dot{p}(r,\theta)}{2p(r,\theta)}
        \qquad
        \text{for all }0<r\le R,
    \end{equation}
    which is exactly \eqref{eqst}. Second, \(D_0(r,\theta)=0\). Third, the lower
    curvature bound attains equality:
    \begin{equation}\label{eq:radial-Bakry-Emery-equality}
        \operatorname{Ric}(\partial_r,\partial_r)
        +\ddot{f}
        -\rho=0.
    \end{equation}

    Since \(D_0=0\), we have
    \[
        |\mathcal{A}|^2=\frac{\omega^2}{n-1}.
    \]
    Substituting this back into the standard Riccati identity
    \[
        \dot{\omega}+|\mathcal{A}|^2
        +\operatorname{Ric}(\partial_r,\partial_r)=0,
    \]
    alongside \eqref{eq:radial-Bakry-Emery-equality}, we obtain
    \[
        \ddot{f}-\rho
        =
        \dot{\omega}
        +
        \frac{\omega^2}{n-1}.
    \]
    Substituting \eqref{eq:equality-mean-curvature} into the right-hand side
    directly provides
    \begin{align*}
        \ddot{f}-\rho
        &=
        \frac{n-1}{2}
        \left(
            \frac{\ddot{p}}{p}
            -
            \frac{\dot{p}^2}{p^2}
        \right)
        +
        \frac{n-1}{4}
        \frac{\dot{p}^2}{p^2}
        \\
        &=
        \frac{n-1}{2}\frac{\ddot{p}}{p}
        -
        \frac{n-1}{4}\frac{\dot{p}^2}{p^2}
        =
        \mathcal{K}_p.
    \end{align*}
    Equivalently, this can be rearranged as
    \[
        \frac{2(\ddot{p}p-\dot{p}^2)}{p^2}
        +
        \frac{\dot{p}^2}{p^2}
        =
        \frac{4(\ddot{f}-\rho)}{n-1},
    \]
    which is precisely the model Riccati equation \eqref{ricatii}.

    Conversely, suppose \eqref{ricatii} and \eqref{eqst} hold on \(B_o(R)\)
    with \(0<R<\operatorname{inj}(o)\). From \eqref{eqst}, we have
    \[
        \partial_r\log J
        =
        \partial_r\log p(r,\theta)^{\frac{n-1}{2}}.
    \]
    Therefore,
    \[
        \partial_r
        \log
        \left(
            \frac{J}{p(r,\theta)^{\frac{n-1}{2}}}
        \right)=0.
    \]
    Using the fact that
    \[
        \frac{J(r,\theta)}{p(r,\theta)^{\frac{n-1}{2}}}\to1
    \]
    as \(r\to0^+\), we conclude that
    \[
        J(r,\theta)=p(r,\theta)^{\frac{n-1}{2}}
    \]
    for all \(0<r\le R\). On the other hand, \eqref{ricatii} is precisely
    \[
        \ddot{f}-\rho-\mathcal{K}_p=0.
    \]
    Thus, the exponential factor in \eqref{eq:correct-volume-comparison} is
    equal to \(1\), and hence
    \[
        \operatorname{Vol}(B_o(r))
        =
        \int_{\mathbb{S}^{n-1}}
        \int_0^r
        p(s,\theta)^{\frac{n-1}{2}}\,ds\,d\theta
        \qquad
        \text{for all }0<r\le R\le \operatorname{inj}(o).
    \]
\end{proof}
If $p(x)$ satisfies \eqref{ricatii} and the conditions of Theorem \ref{thm1.1}, we obtain a volume element comparison:

\begin{lemma}\label{lem:jacobian-ratio-comparison}
Let \((M^n,g)\) be a complete Riemannian manifold satisfying \eqref{RCG}
for some \(f\in C^2(M^n)\) and \(\rho\in C^2(M^n)\). Fix \(o\in M\), and let
\[
    0<R<\operatorname{inj}(o).
\]
Let \(p\in C^2(B_o(R)\setminus\{o\})\) be a positive function satisfying the
assumptions on \(p\) in Theorem \ref{thm1.1}. Assume moreover that \(p\)
satisfies the Riccati equation \eqref{ricatii} on \(B_o(R)\setminus\{o\}\). Then, for every
\(0<r_1<r_2<R\) and every \(\theta\in\mathbb S^{n-1}\),
\begin{equation}\label{eq:jacobian-ratio-comparison}
    \frac{J(r_2,\theta)}{J(r_1,\theta)}
    \le
    \frac{p^{\frac{n-1}{2}}(r_2,\theta)}
         {p^{\frac{n-1}{2}}(r_1,\theta)}.
\end{equation}
\end{lemma}

\begin{proof}
Since \(R<\operatorname{inj}(o)\), the polar coordinates are regular on
\(B_o(R)\setminus\{o\}\). By \eqref{eq:ratio-ineq}, we have
\begin{equation}\label{ratio_deri}
    \partial_r
    \log\left(
        \frac{J(r,\theta)}
             {p^{\frac{n-1}{2}}(r,\theta)}
    \right)
    \le
    \frac{1}{p(r,\theta)}
    \int_0^r
    p(s,\theta)
    \bigl(
        \ddot f(s,\theta)
        -
        \rho(s,\theta)
        -
        \mathcal K_p(s,\theta)
    \bigr)\,ds .
\end{equation}
By \eqref{ricatii},
\[
    \ddot f(r,\theta)-\rho(r,\theta)-\mathcal K_p(r,\theta)=0.
\]
Hence
\begin{equation}\label{JCP}
       \partial_r
    \log\left(
        \frac{J(r,\theta)}
             {p^{\frac{n-1}{2}}(r,\theta)}
    \right)
    \le 0
    \qquad
    \text{for }0<r<R.
\end{equation}
Therefore the function
\[
    r\longmapsto
    \frac{J(r,\theta)}
         {p^{\frac{n-1}{2}}(r,\theta)}
\]
is non-increasing on \((0,R)\). Thus, for \(0<r_1<r_2<R\),
\[
    \frac{J(r_2,\theta)}
         {p^{\frac{n-1}{2}}(r_2,\theta)}
    \le
    \frac{J(r_1,\theta)}
         {p^{\frac{n-1}{2}}(r_1,\theta)}.
\]
This is exactly \eqref{eq:jacobian-ratio-comparison}.
\end{proof}

\begin{remark}
	Since $\Delta r = \partial_r \log J$, inequality \eqref{JCP} can be rewritten as
	\[
	\Delta r \le \frac{(n-1)\dot{p}(r,\theta)}{2p(r,\theta)},
	\]
	providing a Laplacian comparison theorem.
\end{remark}

Following \cite{Pe}, we obtain a volume comparison theorem for geodesic balls:

\begin{theorem}\label{vcieq}
Assume that the hypotheses of Theorem \ref{thm1.1} hold on \(B_o(R)\), with
\(0<R<\operatorname{inj}(o)\). Assume in addition that \(p=p(r)\) is radial.
If \(p\) satisfies \eqref{ricatii} on \(B_o(R)\setminus\{o\}\), then for every
\(0<r_1<r_2\le R\),
\begin{equation}\label{eq:relative-volume-comparison}
    \frac{\operatorname{Vol}(B_o(r_2))}
         {\operatorname{Vol}(B_o(r_1))}
    \le
    \frac{V_p(r_2)}{V_p(r_1)},
\end{equation}
where
\[
    V_p(r)
    =
    \int_{\mathbb S^{n-1}}\int_0^r
    p^{\frac{n-1}{2}}(s)\,ds\,d\theta .
\]
\end{theorem}
\begin{proof}
By Lemma \ref{lem:jacobian-ratio-comparison}, for each fixed
\(\theta\in\mathbb S^{n-1}\), the function
\[
    r\longmapsto
    \frac{J(r,\theta)}{p^{\frac{n-1}{2}}(r)}
\]
is non-increasing on \((0,R]\). Since \(p=p(r)\) is radial, integration over
\(\mathbb S^{n-1}\) gives that
\[
    r\longmapsto
    \frac{\operatorname{Area}(\partial B_o(r))}
         {\int_{\mathbb S^{n-1}}p^{\frac{n-1}{2}}(r)\,d\theta}
\]
is non-increasing. The standard relative volume argument then implies that
\[
    r\longmapsto
    \frac{\operatorname{Vol}(B_o(r))}{V_p(r)}
\]
is non-increasing. Hence, for \(0<r_1<r_2\le R\),
\[
    \frac{\operatorname{Vol}(B_o(r_2))}{V_p(r_2)}
    \le
    \frac{\operatorname{Vol}(B_o(r_1))}{V_p(r_1)}.
\]
This proves \eqref{eq:relative-volume-comparison}.
\end{proof}

Now that we have established the volume comparison for geodesic balls, we turn our attention to the local rigidity. In the classical volume comparison theorems, the equality case implies that the geodesic ball is locally isometric to a ball in the corresponding model space. The following proof rigorously establishes this local metric rigidity for gradient Ricci almost solitons, as stated in Theorem \ref{thm3.1}.

\begin{proof}[Proof of Theorem \ref{thm3.1}]
    First, since \(p\) satisfies \eqref{ricatii}, Theorem \ref{thm1.1} gives
    \[
        \operatorname{Vol}(B_o(r))
        \le
        \int_{\mathbb S^{n-1}}\int_0^r
        p^{\frac{n-1}{2}}(s,\theta)\,ds\,d\theta
    \]
    for every \(0<r\le \min\{R,\operatorname{inj}(o)\}\).

    Suppose that equality holds in \eqref{vbcp} for some fixed
    \(0<r_0\le \min\{R,\operatorname{inj}(o)\}\). Then equality holds in the
    Bakry--\'Emery volume comparison of Theorem \ref{thm1.1} on the geodesic
    ball \(B_o(r_0)\). Equality in \eqref{eq:correct-volume-comparison} gives \(E\equiv0\); hence, by \eqref{eq:defect-E}, \(D_0(s,\theta)=0\) for all
\(0<s\le r_0\) and \(\theta\in\mathbb S^{n-1}\).
    
    Recall that \(D_0=|\mathcal A|^2-\frac{\omega^2}{n-1}\), where
    \[
        \mathcal A=\operatorname{Hess}s\big|_{\{\partial_s\}^{\perp}}
    \]
    is the second fundamental form of the geodesic sphere \(\partial B_s(o)\)
    and \(\omega=\operatorname{tr}\mathcal A=\Delta s\). Consequently, 
   \begin{equation}\label{Hessian expression}
        \operatorname{Hess} s
        =
        \frac{\omega(s,\theta)}{n-1}g_\theta,
   \end{equation}
    where \(g_\theta\) denotes the induced metric on \(\partial B_s(o)\).
    
    By the equality case of Theorem \ref{thm1.1}, we also have
\[
    \omega(s,\theta)=\frac{(n-1)\dot p(s,\theta)}{2p(s,\theta)}.
\]
Substituting this into \eqref{Hessian expression} yields
    \begin{equation}\label{eq:hessian-p}
        \operatorname{Hess} s
        =
        \frac{\dot p(s,\theta)}{2p(s,\theta)}g_\theta.
    \end{equation}

    In polar coordinates \((s,\theta)\) on \(B_o(r_0)\setminus\{o\}\), the
    metric is expressed as
    \[
        g=ds^2+g_{ij}(s,\theta)d\theta^i d\theta^j.
    \]
    Utilizing 
    \[
        \mathcal L_{\partial_s}g=2\operatorname{Hess}s,
    \]
    equation \eqref{eq:hessian-p} translates to 
    \begin{equation}\label{eq:metric_evol}
        \frac{\partial}{\partial s}g_{ij}(s,\theta)
        =
        \frac{\dot p(s,\theta)}{p(s,\theta)}g_{ij}(s,\theta).
    \end{equation}
    
    For each fixed \(\theta\), integrating this ordinary differential equation
    from \(s=\epsilon\) to \(s=r\le r_0\) gives
    \begin{equation}\label{eq:metric_int}
        g_{ij}(r,\theta)
        =
        \frac{p(r,\theta)}{p(\epsilon,\theta)}g_{ij}(\epsilon,\theta).
    \end{equation}

    As \(M\) is smooth at the pole \(o\), the metric in polar coordinates
    satisfies the standard asymptotic expansion
    \[
        g_{ij}(\epsilon,\theta)
        =
        \epsilon^2h_{ij}(\theta)+o(\epsilon^2)
        \qquad\text{as }\epsilon\to0^+,
    \]
    where \(h_{ij}\) is the standard metric on \(\mathbb S^{n-1}\).
   By the asymptotic assumption \(p(r,\theta)=r^2+O(r^4)\) near \(r=0\), we have
\[
    p(\varepsilon,\theta)=\varepsilon^2+O(\varepsilon^4).
\]
    Taking \(\epsilon\to0^+\) in \eqref{eq:metric_int}, we obtain
    \[
        g_{ij}(r,\theta)
        =
        p(r,\theta)h_{ij}(\theta).
    \]
    Hence, the metric on \(B_o(r_0)\) splits as the twisted product metric
    \[
        g=dr^2+p(r,\theta)g_{\mathbb S^{n-1}}.
    \]
    This concludes the proof that \(B_o(r_0)\) is locally isometric to the
    corresponding model ball.

    Conversely, if \(B_o(r_0)\) is isometric to such a model ball, then
\(J(r,\theta)=p^{(n-1)/2}(r,\theta)\) for \(0<r\le r_0\), and hence equality
holds in \eqref{vbcp} with \(r=r_0\).
\end{proof}
\begin{remark}
Taking \(\rho=(n-1)H\) and \(p(r,\theta)=\mathrm{sn}_H^2(r)\), we have
\[
    \mathcal K_p=-(n-1)H.
\]
Hence
\[
    \ddot f-\rho-\mathcal K_p=\ddot f.
\]
If \(\dot f\ge -a\), then \eqref{eq:ratio-ineq} gives, for
\(0<r_1<r_2<R\) and \(R<\frac{\pi}{2\sqrt H}\) when \(H>0\),
\[
    J(r_2,\theta)
    \le
    e^{a(r_2-r_1)+f(r_2,\theta)-f(r_1,\theta)}
    \frac{\mathrm{sn}_H^{\,n-1}(r_2)}
         {\mathrm{sn}_H^{\,n-1}(r_1)}
    J(r_1,\theta).
\]
Equivalently,
\[
    e^{-f(r_2,\theta)}J(r_2,\theta)
    \le
    e^{a(r_2-r_1)}
    \frac{\mathrm{sn}_H^{\,n-1}(r_2)}
         {\mathrm{sn}_H^{\,n-1}(r_1)}
    e^{-f(r_1,\theta)}J(r_1,\theta).
\]
Integrating this weighted density comparison gives
\[
    \frac{\mathrm{Vol}_f(B_o(r_2))}
         {\mathrm{Vol}_f(B_o(r_1))}
    \le
    e^{aR}
    \frac{V_H(r_2)}{V_H(r_1)},
\]
where
\[
    \mathrm{Vol}_f(B_o(r))
    =
    \int_0^r\int_{\mathbb S^{n-1}}
    e^{-f(k,\theta)}J(k,\theta)\,d\theta\,dk
\]
and
\[
    V_H(r)
    =
    \int_0^r\int_{\mathbb S^{n-1}}
    \mathrm{sn}_H^{\,n-1}(k)\,d\theta\,dk .
\]

If instead \(\rho=(n-1)H\) and \(|f|\le c\), then applying
\eqref{eq:ratio-ineq} with \(p(r,\theta)=\mathrm{sn}_H^2(r)\) gives, for
\(0<r_1<r_2<R\) and \(R<\frac{\pi}{4\sqrt H}\) when \(H>0\),
\[
    J(r_2,\theta)
    \le
    e^{f(r_2,\theta)-f(r_1,\theta)}
    \frac{\mathrm{sn}_H^{\,n-1+4c}(r_2)}
         {\mathrm{sn}_H^{\,n-1+4c}(r_1)}
    J(r_1,\theta).
\]
Thus
\[
    e^{-f(r_2,\theta)}J(r_2,\theta)
    \le
    \frac{\mathrm{sn}_H^{\,n-1+4c}(r_2)}
         {\mathrm{sn}_H^{\,n-1+4c}(r_1)}
    e^{-f(r_1,\theta)}J(r_1,\theta).
\]
Integrating again yields
\[
    \frac{\mathrm{Vol}_f(B_o(r_2))}
         {\mathrm{Vol}_f(B_o(r_1))}
    \le
    \frac{V_H^{\,n+4c}(r_2)}
         {V_H^{\,n+4c}(r_1)},
\]
where
\[
    V_H^{\,n+4c}(r)
    =
    \int_0^r\int_{\mathbb S^{n-1}}
    \mathrm{sn}_H^{\,n-1+4c}(k)\,d\theta\,dk .
\]
Therefore, by choosing special test functions \(p\), the comparison formula
recovers the volume comparison results of Wei--Wylie \cite{WW}.
\end{remark}

\section{Rigidity results}
\label{sec:rigidity}

In this section, we prove Theorem~\ref{hhg}, establishing the global rigidity of complete noncompact Riemannian manifolds satisfying the curvature condition
\eqref{RCG}, and in particular gradient Ricci almost solitons, when the volume growth attains its upper bound for all $R$.

We first prove Theorem~\ref{hhg} by using the global volume comparison.

\begin{proof}[Proof of Theorem \ref{hhg}]
    To establish global rigidity, we utilize the global volume equality to prove that the cut locus of the manifold $M^n$ must be empty.

By hypothesis, the test function \(p(r,\theta)\) is defined globally and
satisfies \(p(r,\theta)>0\) for all \(r>0\). The assumption \(p(r,\theta)>0\) for all \(r>0\) ensures that the model density \(p^{(n-1)/2}(r,\theta)\) is strictly positive for all \(r>0\).

   Let \(c:\mathbb S^{n-1}\to(0,\infty]\) be the cut function based at \(o\). For
each \(\theta\in\mathbb S^{n-1}\), \(c(\theta)\) is the supremum of all
\(s>0\) such that the radial geodesic \(t\mapsto\exp_o(t\theta)\) is minimizing
on \([0,s]\). Equivalently, if a cut point occurs in the direction \(\theta\),
then \(c(\theta)\) is the distance from \(o\) to that cut point; if no cut point
occurs, then \(c(\theta)=\infty\). As established in Petersen \cite[Section 5.7.3]{Pe}, the segment domain $\operatorname{seg}(o) \subset T_o M$ is a closed star-shaped subset with an open star-interior $\operatorname{seg}^0(o)$. By the standard properties of the segment domain, the function
\(c(\theta)\) is upper semicontinuous. Hence the set
\(\{\theta:c(\theta)<a\}\) is open for every \(a>0\). Specifically, to establish upper semi-continuity, consider a sequence $\theta_i \to \theta$ in $\mathbb{S}^{n-1}$. We can choose lengths $t_i \le c(\theta_i)$ such that $t_i \to \limsup_{i \to \infty} c(\theta_i)$. By definition, the tangent vectors $t_i\theta_i$ belong to the segment domain $\operatorname{seg}(o)$. Because $\operatorname{seg}(o)$ is a closed subset of $T_oM$, their limit vector $(\limsup_{i \to \infty} c(\theta_i))\theta$ must also belong to $\operatorname{seg}(o)$. This forces $\limsup_{i \to \infty} c(\theta_i) \le c(\theta)$, proving upper semi-continuity. Conversely, the openness of the interior $\operatorname{seg}^0(o)$ implies that for any length $t < c(\theta)$, the open ball in $T_o M$ centered at $t\theta$ is completely contained in $\operatorname{seg}^0(o)$; thus, rays in sufficiently close directions remain minimizing at least up to length $t$, establishing lower semi-continuity. The combination of upper and lower semi-continuity rigorously guarantees that $c(\theta)$ is continuous on the unit sphere wherever it takes finite values.

  Since \(p\) satisfies \eqref{ricatii}, inequality \eqref{J controll} on the regular set \(0<s<c(\theta)\)
reduces to
\[
    J(s,\theta)\le p^{\frac{n-1}{2}}(s,\theta).
\]
Hence,
\[
\begin{aligned}
\operatorname{Vol}(B_o(R))
&=
\int_{\mathbb S^{n-1}}
\int_0^{\min\{R,c(\theta)\}}
J(s,\theta)\,ds\,d\theta \\
&\le
\int_{\mathbb S^{n-1}}
\int_0^{\min\{R,c(\theta)\}}
p^{\frac{n-1}{2}}(s,\theta)\,ds\,d\theta \\
&\le
\int_{\mathbb S^{n-1}}
\int_0^R
p^{\frac{n-1}{2}}(s,\theta)\,ds\,d\theta.
\end{aligned}
\]

    We claim that $c(\theta) = \infty$ for all $\theta \in \mathbb{S}^{n-1}$. Suppose conversely, that there exists a direction $\theta_0 \in \mathbb{S}^{n-1}$ such that $c(\theta_0) = R_0 < \infty$. By the continuity of $c(\theta)$ at $\theta_0$, for a sufficiently small fixed $\varepsilon > 0$, the set $V = \{ \theta \in \mathbb{S}^{n-1} \mid c(\theta) < R_0 + \varepsilon \}$ is a non-empty open neighborhood of $\theta_0$ in $\mathbb{S}^{n-1}$.

   Choose \(R>R_0+\varepsilon\), and define
\[
    F(\theta)
    =
    \int_0^R p(s,\theta)^{\frac{n-1}{2}}\,ds
    -
    \int_0^{c(\theta)} p(s,\theta)^{\frac{n-1}{2}}\,ds
    =
    \int_{c(\theta)}^R p(s,\theta)^{\frac{n-1}{2}}\,ds .
\]
    For all $\theta \in V$, since $c(\theta) < R_0 + \varepsilon < R$, the integration interval $[c(\theta), R]$ has strictly positive length. Moreover, the model density $p^{\frac{n-1}{2}}(s, \theta)$ is strictly positive for $s > 0$. Because $c(\theta)$ is continuous on $V$, $F(\theta)$ is a well-defined, continuous, and strictly positive function on $V$.

    Since $V$ is a non-empty open set in the smooth manifold $\mathbb{S}^{n-1}$, its standard spherical Lebesgue measure is strictly positive ($\mu(V) > 0$). Integrating the strictly positive continuous function $F(\theta)$ over $V$ yields a strictly positive result:
    $$
        \int_V \int_0^{c(\theta)} p^{\frac{n-1}{2}}(s, \theta) \,ds \,d\theta < \int_V \int_0^R p^{\frac{n-1}{2}}(s, \theta) \,ds \,d\theta.
    $$
    On the complement $\mathbb{S}^{n-1} \setminus V$, we trivially have $\min(R, c(\theta)) \le R$, implying that 
     $$
        \int_{\mathbb{S}^{n-1} \setminus V} \int_0^{\min\{R,c(\theta)\}} J(s,\theta) \,ds \,d\theta \le \int_{\mathbb{S}^{n-1} \setminus V} \int_0^R p^{\frac{n-1}{2}}(s, \theta) \,ds \,d\theta.
    $$
    Combining these two regions yields the strict global inequality:
 \begin{equation}
\begin{aligned}
\operatorname{Vol}_{M^n}(B_o(R))
&=
\int_V
\int_0^{c(\theta)}
J(s,\theta)\,ds\,d\theta
+
\int_{\mathbb S^{n-1}\setminus V}
\int_0^{\min\{R,c(\theta)\}}
J(s,\theta)\,ds\,d\theta                                      \\
&\le
\int_V
\int_0^{c(\theta)}
p^{\frac{n-1}{2}}(s,\theta)\,ds\,d\theta
+
\int_{\mathbb S^{n-1}\setminus V}
\int_0^R
p^{\frac{n-1}{2}}(s,\theta)\,ds\,d\theta                       \\
&<
\int_V
\int_0^R
p^{\frac{n-1}{2}}(s,\theta)\,ds\,d\theta
+
\int_{\mathbb S^{n-1}\setminus V}
\int_0^R
p^{\frac{n-1}{2}}(s,\theta)\,ds\,d\theta                       \\
&=
\int_{\mathbb S^{n-1}}
\int_0^R
p^{\frac{n-1}{2}}(s,\theta)\,ds\,d\theta .
\end{aligned}
\end{equation}
    This contradicts the hypothesis that volume equality holds for all $R > 0$. Consequently, $c(\theta) = \infty$ for all $\theta \in \mathbb{S}^{n-1}$, proving that the cut locus $\operatorname{Cut}(o)$ is empty.

   Since \(c(\theta)=\infty\) for every \(\theta\in\mathbb S^{n-1}\), the cut
locus of \(o\) is empty. By the standard characterization of the cut locus,
\(\exp_o\) has no singular points and is one-to-one. By the Hopf--Rinow theorem,
\(\exp_o\) is onto. Hence
\[
    \exp_o:T_oM\to M
\]
is a global diffeomorphism. In particular, \(M^n\) is diffeomorphic to
\(\mathbb R^n\), and therefore \(\pi_k(M^n)=0\) for all \(k\ge1\).

Moreover, \(\operatorname{inj}(o)=\infty\), so Theorem \ref{thm3.1} applies to
\(B_o(R)\) for every \(R>0\). By \eqref{eq:global-volume-equality}, equality
holds in the local comparison on every \(B_o(R)\). Therefore each \(B_o(R)\) is
isometric to the corresponding model ball. Letting \(R\to\infty\), we obtain
the global isometry.
\end{proof}

For gradient Ricci almost solitons with prescribed volume growth, we obtain:

\begin{corollary}\label{thm3.2}
Let \((M^n,g,f,\rho)\) be a complete noncompact gradient Ricci almost soliton
satisfying the assumptions of Theorem \ref{hhg}. Let \(K\ge0\). Suppose that
\[
    p(r)=\mathrm{sn}_{-K}^2(r)
\]
satisfies \eqref{ricatii} for all \(r\ge0\), and that for every \(R>0\),
\[
    \operatorname{Vol}(B_o(R))
    =
    \int_{\mathbb S^{n-1}}\int_0^R
    \mathrm{sn}_{-K}^{\,n-1}(s)\,ds\,d\theta .
\]
Then \((M^n,g)\) is globally isometric to the simply connected space form of
constant sectional curvature \(-K\).
\end{corollary}

\begin{proof}
By assumption, the model function
\[
    p(r)=\mathrm{sn}_{-K}^2(r)
\]
solves \eqref{ricatii} for all \(r\ge0\). Moreover, the volume equality above
is exactly the model volume equality required in Theorem \ref{hhg}. Hence
Theorem \ref{hhg} applies, and \((M^n,g)\) is globally isometric to the polar
model
\[
    dr^2+\mathrm{sn}_{-K}^2(r)g_{\mathbb S^{n-1}}.
\]
This is the standard metric on the simply connected space form of constant
sectional curvature \(-K\).
\end{proof}

\section{Estimates of the potential function and volume growth for shrinking almost solitons}
\label{sec:potential_volume}

In \cite{CZ}, the authors established growth estimates for the potential function of shrinking gradient Ricci solitons. For gradient Ricci almost solitons, a lower bound for the scalar curvature \(R\) is not automatic. In this section, we extend the results of \cite{CZ} by assuming that the scalar curvature is bounded below by a constant $R_{min}$. By combining this with an integral condition on the radial derivative of $\rho(x)$, we derive the quadratic growth estimates for the potential function $f$. These estimates are essential for our subsequent analysis of volume growth and rigidity for shrinking gradient Ricci almost solitons.

First, the scalar curvature of gradient Ricci almost solitons satisfies an equation analogous to that in \cite{CZ}.

\begin{proposition}\label{prop4.1}
	The scalar curvature $R$ of a gradient Ricci almost soliton satisfies
	\[
	\partial_i R = 2R_{ij}\partial_j f + 2(n-1)\rho_i,
	\]
	where $R_{ij}$ denotes the Ricci curvature.
\end{proposition}

\begin{proof}
	From the gradient Ricci almost soliton equation, the derivatives of the Ricci curvature satisfy
	\begin{align*}
		\partial_i R_{jk} &= -\nabla_i\nabla_j(\nabla_k f) + \rho_i g_{jk}, \\
		\partial_j R_{ik} &= -\nabla_j\nabla_i(\nabla_k f) + \rho_j g_{ik}.
	\end{align*}
	Subtracting these two equations gives
	\[
	\partial_i R_{jk} - \partial_j R_{ik} = R_{ijlk}\partial_l f + \rho_i g_{jk} - \rho_j g_{ik}.
	\]
	Taking the trace in $j$ and $k$ yields
	\[
	\partial_i R - \partial_j R_{ij} = R_{il}\partial_l f + (n-1)\rho_i.
	\]
	Using the second Bianchi identity for Ricci curvature, we obtain
	\begin{equation}\label{eq5.3}
		\partial_i R = 2R_{ij}\partial_j f + 2(n-1)\rho_i. \qedhere
	\end{equation}
\end{proof}

\begin{remark}
	In this proof and subsequent discussions, we use Einstein summation over repeated indices. The indices $i, j, k, l$ denote normal coordinates.
\end{remark}

Substituting \(R_{ij}=\rho g_{ij}-f_{ij}\) into \eqref{eq5.3}, we obtain the
following identity.

\begin{corollary}\label{cor4.3}
Let \((M^n,g,f,\rho)\) be a gradient Ricci almost soliton. Define
\begin{equation}\label{eq5.4}
      F:=R+|\nabla f|^2.
\end{equation}
Then
\begin{equation}\label{eq:F-gradient}
    \nabla F=2\rho \nabla f+2(n-1)\nabla \rho .
\end{equation}
In particular, along every radial geodesic \(\gamma_\theta(r)\) from \(o\),
\begin{equation}\label{eq:F-radial}
    \dot F(r,\theta)
    =
    2\rho(r,\theta)\dot f(r,\theta)
    +
    2(n-1)\dot \rho(r,\theta).
\end{equation}
\end{corollary}

\begin{proof}
By Proposition \ref{prop4.1},
\[
    \nabla_iR
    =
    2R_{ij}\nabla_j f
    +
    2(n-1)\nabla_i\rho .
\]
Since \(R_{ij}=\rho g_{ij}-f_{ij}\), we get
\[
\begin{aligned}
    \nabla_i(R+|\nabla f|^2)
    &=
    \nabla_iR+2f_{ij}\nabla_j f                                      \\
    &=
    2(\rho g_{ij}-f_{ij})\nabla_j f
    +
    2(n-1)\nabla_i\rho
    +
    2f_{ij}\nabla_j f                                                  \\
    &=
    2\rho \nabla_i f
    +
    2(n-1)\nabla_i\rho .
\end{aligned}
\]
This proves \eqref{eq:F-gradient}. Taking the radial derivative gives
\eqref{eq:F-radial}.
\end{proof}

\begin{theorem}\label{thm4.5}
Let \((M^n,g,f,\rho)\) be a shrinking gradient Ricci almost soliton with
\[
    0<K_2\le \rho(x)\le K_1.
\]
Fix \(o\in M\), and set \(r(x)=d(o,x)\). Assume that the scalar curvature satisfies
\[
    R\ge R_{\min}>-\infty,
\]
and that there exists a constant \(K\ge0\) such that, along every radial
geodesic from \(o\),
\begin{equation}\label{eq:rho-f-integral-lower}
    \int_0^r \dot\rho(s,\theta) f(s,\theta)\,ds\ge -K
    \qquad
    \text{for all }r\ge0.
\end{equation}
Then the potential function \(f\) satisfies
\begin{equation}\label{eq:upper-potential-estimate}
    f(x)\le
    \left(
        \sqrt{\frac{K_1}{2}}\,r(x)+C_2
    \right)^2
\end{equation}
for all \(x\in M\), where \(C_2\) depends only on
\(n,K,K_1,K_2,R_{\min},\rho(o),f(o),F(o)\).
\end{theorem}

\begin{proof}
By \eqref{eq5.4} and the scalar curvature lower bound \(R\ge R_{\min}\), we have
\[
    |\nabla f|^2=F-R\le F-R_{\min}.
\]
Integrating \eqref{eq:F-radial} from \(0\) to \(r\), we obtain
\[
\begin{aligned}
    F(r,\theta)-F(o)
    &=
    \int_0^r
    \bigl(
        2\rho \dot f
        +
        2(n-1)\dot\rho
    \bigr)\,ds                                                   \\
    &=
    \bigl[2\rho(s,\theta)f(s,\theta)\bigr]_0^r
    -
    2\int_0^r \dot\rho(s,\theta)f(s,\theta)\,ds                  \\
    &\quad
    +
    2(n-1)\bigl(\rho(r,\theta)-\rho(o)\bigr).
\end{aligned}
\]
Using \eqref{eq:rho-f-integral-lower} and \(\rho\le K_1\), we get
\[
\begin{aligned}
    F(r,\theta)
    &\le
    2\rho(r,\theta)f(r,\theta)
    -
    2\rho(o)f(o)
    +
    F(o)
    +
    2K
    +
    2(n-1)\bigl(\rho(r,\theta)-\rho(o)\bigr)                    \\
    &\le
    2\rho(r,\theta)f(r,\theta)
    +
    C_F,
\end{aligned}
\]
where
\[
    C_F
    :=
    F(o)-2\rho(o)f(o)+2K+2(n-1)(K_1-\rho(o)).
\]
Therefore
\[
    |\nabla f|^2
    \le
    F-R_{\min}
    \le
    2\rho f+C_0,
\]
where
\[
    C_0:=C_F-R_{\min}
    =
    F(o)-2\rho(o)f(o)+2K+2(n-1)(K_1-\rho(o))-R_{\min}.
\]
Thus
\begin{equation}\label{eq:grad-f-rho}
    |\nabla f|^2
    \le
    2\rho f+C_0.
\end{equation}

Let
\[
    C_0^+:=\max\{C_0,0\}.
\]
Since \(|\nabla f|^2\ge0\) and \(\rho\ge K_2>0\), \eqref{eq:grad-f-rho}
implies
\[
    f\ge -\frac{C_0^+}{2K_2}.
\]
Choose
\[
    B:=1+\frac{C_0^+}{2K_2}.
\]
Then \(f+B>0\). Moreover, from \eqref{eq:grad-f-rho},
\[
\begin{aligned}
    |\nabla f|^2
    &\le
    2\rho f+C_0                                                   \\
    &=
    2\rho(f+B)-2\rho B+C_0                                        \\
    &\le
    2K_1(f+B)+C_0^+ .
\end{aligned}
\]
Set
\[
    C_1:=B+\frac{C_0^+}{2K_1}.
\]
Then \(f+C_1>0\), and
\begin{equation}\label{eq:gradient-f-C1}
    |\nabla f|^2
    \le
    2K_1(f+C_1).
\end{equation}
Hence
\[
    \left|\nabla\sqrt{f+C_1}\right|^2
    =
    \frac{|\nabla f|^2}{4(f+C_1)}
    \le
    \frac{K_1}{2}.
\]
Along a minimizing geodesic from \(o\) to \((r,\theta)\), this gives
\[
    \sqrt{f(r,\theta)+C_1}
    \le
    \sqrt{f(o)+C_1}
    +
    \sqrt{\frac{K_1}{2}}\,r.
\]
Set
\[
    C_2:=\sqrt{f(o)+C_1}.
\]
Then
\[
    f(r,\theta)
    \le
    \left(
        \sqrt{\frac{K_1}{2}}\,r+C_2
    \right)^2,
\]
which proves \eqref{eq:upper-potential-estimate}.
\end{proof}
Following the same argument as in \cite{CZ}, we can derive a lower bound for the potential function using its upper bound and the lower bound of $\rho(x)$.

\begin{theorem}\label{thm4.6}
Let \((M^n,g,f,\rho)\) be a complete noncompact shrinking gradient Ricci
almost soliton. Fix \(o\in M\), and set \(r(x)=d(o,x)\). Assume that
\[
    0<K_2\le \rho(x)\le K_1
\]
and that the scalar curvature satisfies
\[
    R\ge R_{\min}>-\infty.
\]
Suppose that there exists a constant \(K\ge0\) such that, along every radial
minimizing geodesic from \(o\),
\[
    \int_0^r \dot\rho(s,\theta)f(s,\theta)\,ds\ge -K
    \qquad
    \text{for all }r\ge0.
\]
Then there exist constants \(C_3>0\) and \(r_0>0\) such that, for all
\(x\in M\) with \(r(x)\ge r_0\),
\begin{equation}\label{eq:two-sided-potential-growth}
    \frac{\bigl(K_2r(x)-C_3\bigr)^2}{2K_1}
    \le
    f(x)
    \le
    \left(
        \sqrt{\frac{K_1}{2}}\,r(x)+C_2
    \right)^2,
\end{equation}
where \(C_2\) is the constant in Theorem \ref{thm4.5}.
\end{theorem}

\begin{proof}
Let \(x\in M\) with \(s_0=r(x)>2\), and let
\(\gamma:[0,s_0]\to M\) be a unit-speed minimizing geodesic from \(o\) to \(x\).
Set \(X(t)=\dot\gamma(t)\). The second variation of arc length gives
\[
    \int_0^{s_0}\phi^2 \operatorname{Ric}(X,X)\,dt
    \le
    (n-1)\int_0^{s_0}|\dot\phi(t)|^2\,dt
\]
for any piecewise smooth function \(\phi\) with
\(\phi(0)=\phi(s_0)=0\). Following \cite{Ha1}, choose
\[
\phi(t)=
\begin{cases}
    t, & t\in[0,1],\\
    1, & t\in[1,s_0-1],\\
    s_0-t, & t\in[s_0-1,s_0].
\end{cases}
\]
Then
\[
    \int_0^{s_0}\phi^2 \operatorname{Ric}(X,X)\,dt
    \le 2(n-1).
\]
Therefore
\begin{align*}
    \int_1^{s_0-1}\operatorname{Ric}(X,X)\,dt
    &=
    \int_1^{s_0-1}\phi^2\operatorname{Ric}(X,X)\,dt  \\
    &=
    \int_0^{s_0}\phi^2\operatorname{Ric}(X,X)\,dt
    -
    \int_0^1\phi^2\operatorname{Ric}(X,X)\,dt          \\
    &\quad
    -
    \int_{s_0-1}^{s_0}\phi^2\operatorname{Ric}(X,X)\,dt \\
    &\le
    2(n-1)
    +
    \max_{B_o(1)}|\operatorname{Ric}|
    -
    \int_{s_0-1}^{s_0}\phi^2\operatorname{Ric}(X,X)\,dt.
\end{align*}

Using the almost soliton equation along \(\gamma\),
\[
    \nabla_X\dot f=\rho-\operatorname{Ric}(X,X),
\]
we get
\begin{align*}
    \dot f(\gamma(s_0-1))-\dot f(\gamma(1))
    &=
    \int_1^{s_0-1}\nabla_X\dot f\,dt                         \\
    &=
    \int_1^{s_0-1}\rho(\gamma(t))\,dt
    -
    \int_1^{s_0-1}\operatorname{Ric}(X,X)\,dt                  \\
    &\ge
    K_2(s_0-2)
    -
    2(n-1)
    -
    \max_{B_o(1)}|\operatorname{Ric}|                          \\
    &\quad
    +
    \int_{s_0-1}^{s_0}\phi^2\operatorname{Ric}(X,X)\,dt .
\end{align*}
Rearranging, we obtain
\begin{equation}\label{eq:ric_upper}
\begin{aligned}
    \int_{s_0-1}^{s_0}\phi^2\operatorname{Ric}(X,X)\,dt
    &\le
    \dot f(\gamma(s_0-1))-\dot f(\gamma(1))                    \\
    &\quad
    -K_2(s_0-2)
    +2(n-1)
    +\max_{B_o(1)}|\operatorname{Ric}|.
\end{aligned}
\end{equation}

On the other hand, using
\[
    \operatorname{Ric}(X,X)=\rho-\nabla_X\dot f,
\]
we evaluate the same integral directly. Since
\(\phi(s_0)=0\), \(\phi(s_0-1)=1\), and
\(\dot\phi=-1\) on \([s_0-1,s_0]\), integration by parts gives
\begin{align*}
    \int_{s_0-1}^{s_0}\phi^2\operatorname{Ric}(X,X)\,dt
    &=
    \int_{s_0-1}^{s_0}\phi^2\rho(\gamma(t))\,dt
    -
    \int_{s_0-1}^{s_0}\phi^2\nabla_X\dot f(\gamma(t))\,dt \\
    &\ge
    \frac{K_2}{3}
    +
    \left[-\phi^2\dot f(\gamma(t))\right]_{s_0-1}^{s_0}
    +
    \int_{s_0-1}^{s_0}2\phi\dot\phi\,\dot f(\gamma(t))\,dt \\
    &=
    \frac{K_2}{3}
    +
    \dot f(\gamma(s_0-1))
    -
    2\int_{s_0-1}^{s_0}\phi\,\dot f(\gamma(t))\,dt.
\end{align*}
Combining this lower bound with \eqref{eq:ric_upper}, the term
\(\dot f(\gamma(s_0-1))\) cancels. Hence
\[
    2\int_{s_0-1}^{s_0}\phi\,\dot f(\gamma(t))\,dt
    \ge
    K_2s_0-c_0,
\]
where
\[
    c_0
    :=
    \frac{5}{3}K_2
    +2(n-1)
    +\max_{B_o(1)}|\operatorname{Ric}|
    +\max_{B_o(1)}|\nabla f|.
\]
Indeed, we used
\[
    \dot f(\gamma(1))\ge -\max_{B_o(1)}|\nabla f|.
\]

Since
\[
    \int_{s_0-1}^{s_0}\phi(t)\,dt=\frac12,
\]
there exists \(t^*\in[s_0-1,s_0]\) such that
\[
    \dot f(\gamma(t^*))\ge K_2s_0-c_0.
\]
Taking \(s_0\) sufficiently large so that \(K_2s_0-c_0>0\), and using the
gradient estimate
\[
    |\nabla f|^2\le 2K_1(f+C_1),
\]
we obtain
\[
    \sqrt{2K_1\bigl(f(\gamma(t^*))+C_1\bigr)}
    \ge
    K_2s_0-c_0.
\]

Now set
\[
    y(t):=\sqrt{f(\gamma(t))+C_1}.
\]
By the same gradient estimate,
\[
    |\dot y(t)|
    =
    \frac{|\dot f(\gamma(t))|}
         {2\sqrt{f(\gamma(t))+C_1}}
    \le
    \sqrt{\frac{K_1}{2}}.
\]
Thus \(y\) is uniformly Lipschitz along \(\gamma\). Since
\(s_0-t^*\le1\), we have
\[
    y(s_0)
    \ge
    y(t^*)-\int_{t^*}^{s_0}|\dot y(t)|\,dt
    \ge
    y(t^*)-\sqrt{\frac{K_1}{2}}.
\]
Therefore, for \(s_0\) sufficiently large,
\[
    \sqrt{f(\gamma(s_0))+C_1}
    \ge
    \frac{K_2s_0-c_0}{\sqrt{2K_1}}
    -
    \sqrt{\frac{K_1}{2}}
    =
    \frac{K_2s_0-(c_0+K_1)}{\sqrt{2K_1}}.
\]
Thus
\[
    f(\gamma(s_0))+C_1
    \ge
    \frac{(K_2s_0-(c_0+K_1))^2}{2K_1}.
\]
After increasing the constant and taking \(s_0\ge r_0\) sufficiently large, we
can choose \(C_3>0\) such that
\[
    f(\gamma(s_0))
    \ge
    \frac{(K_2s_0-C_3)^2}{2K_1}.
\]
Since \(s_0=r(x)\), this gives
\[
    f(x)
    \ge
    \frac{(K_2r(x)-C_3)^2}{2K_1}
\]
for all \(r(x)\ge r_0\). Together with the upper bound in Theorem
\ref{thm4.5}, this proves the theorem.
\end{proof}
\begin{remark}
The scalar curvature lower bound assumed above is consistent with the existing
theory of Ricci almost solitons. For example, Pigola, Rigoli, Rimoldi, and
Setti \cite{PRR} obtained scalar curvature nonnegativity results under natural
additional assumptions, such as the superharmonicity of the soliton function
\(\rho\). Thus, in the cases covered by those results, the hypothesis
\(R\ge R_{\min}\) in Theorems~\ref{thm4.5} and~\ref{thm4.6} is naturally
satisfied, for instance with \(R_{\min}=0\).
\end{remark}

We now prove Theorem \ref{thm1.3}. 

\begin{proof}[Proof of Theorem \ref{thm1.3}]
All pointwise computations are first made on the regular set
\(0<r<c(\theta)\), where \(c(\theta)\) denotes the cut time of \(o\) in the
direction \(\theta\).

Using \(p=r^2\) as the test function in the proof of Theorem \ref{thm1.1}, we
obtain
\begin{equation}\label{eq:r2-completion}
\begin{aligned}
    r^2\omega(r,\theta)
    \le
    (n-1)r
    -
    \int_0^r
    \frac{\bigl(t\omega(t,\theta)-(n-1)\bigr)^2}{n-1}\,dt
    -
    \int_0^r
    t^2\operatorname{Ric}(\partial_t,\partial_t)\,dt .
\end{aligned}
\end{equation}
Let
\[
    L(r,\theta)=\log\left(\frac{J(r,\theta)}{r^{n-1}}\right).
\]
Then \eqref{eq:r2-completion} implies
\begin{equation}\label{eq:L-prime}
    \dot{L}(r,\theta)
    \le
    -
    \frac1{r^2}
    \int_0^r
    t^2\operatorname{Ric}(\partial_t,\partial_t)\,dt .
\end{equation}
Integrating from \(0\) to \(r\) and using
\[
    \lim_{r\to0^+}L(r,\theta)=0,
\]
we get
\[
    L(r,\theta)
    \le
    -
    \int_0^r
    \frac1{s^2}
    \left(
        \int_0^s
        t^2\operatorname{Ric}(\partial_t,\partial_t)\,dt
    \right)ds.
\]
By integration by parts,
\[
\begin{aligned}
    \int_0^r
    \frac1{s^2}
    \left(
        \int_0^s
        t^2\operatorname{Ric}(\partial_t,\partial_t)\,dt
    \right)ds
    &=
    -\frac1r
    \int_0^r
    t^2\operatorname{Ric}(\partial_t,\partial_t)\,dt        \\
    &\quad
    +
    \int_0^r
    t\operatorname{Ric}(\partial_t,\partial_t)\,dt .
\end{aligned}
\]
Here the boundary term at \(s=0\) vanishes since
\[
    \lim_{s\to0^+}
    \frac{\int_0^s t^2\operatorname{Ric}(\partial_t,\partial_t)\,dt}{s}
    =
    0.
\]
Thus
\[
    L(r,\theta)
    \le
    \frac1r
    \int_0^r
    t^2\operatorname{Ric}(\partial_t,\partial_t)\,dt
    -
    \int_0^r
    t\operatorname{Ric}(\partial_t,\partial_t)\,dt .
\]
Since
\[
    (rL)'(r,\theta)=L(r,\theta)+r\dot L(r,\theta),
\]
combining the last inequality with \eqref{eq:L-prime} gives
\begin{equation}\label{eq:rL}
    \partial_r(rL)(r,\theta)
    \le
    -
    \int_0^r
    t\operatorname{Ric}(\partial_t,\partial_t)\,dt .
\end{equation}

Along the radial geodesic \(\gamma_\theta\), write
\[
    f(t)=f(\gamma_\theta(t)).
\]
Since
\[
    \operatorname{Ric}+\nabla^2f=\rho g
    \qquad\text{and}\qquad
    \rho\ge\frac12,
\]
we have
\[
    \operatorname{Ric}(\partial_t,\partial_t)
    =
    \rho-f''(t)
    \ge
    \frac12-f''(t).
\]
Therefore
\begin{equation}\label{eq:ricci-to-f}
\begin{aligned}
    -
    \int_0^r
    t\operatorname{Ric}(\partial_t,\partial_t)\,dt
    &\le
    -
    \int_0^r
    t\left(\frac12-f''(t)\right)\,dt                         \\
    &=
    -\frac{r^2}{4}+rf'(r)-f(r)+f(o).
\end{aligned}
\end{equation}
The last expression has the decomposition
\begin{equation}\label{eq:square-split}
\begin{aligned}
    -\frac{r^2}{4}+rf'(r)-f(r)+f(o)
    &=
    f(o)-\bigl(f-|\nabla f|^2\bigr)(r,\theta)                  \\
    &\quad
    -
    \left(\frac r2-f'(r)\right)^2
    -
    \left(|\nabla f|^2(r,\theta)-(f'(r))^2\right).
\end{aligned}
\end{equation}
Since
\[
    |\nabla f|^2-(f')^2\ge0,
\]
it follows from \eqref{eq:rL}--\eqref{eq:square-split} that
\[
    (rL)'(r,\theta)
    \le
    f(o)-\bigl(f-|\nabla f|^2\bigr)(r,\theta).
\]
Integrating from \(0\) to \(r\), we obtain
\[
    L(r,\theta)
    \le
    f(o)
    -
    \frac1r
    \int_0^r
    \bigl(f-|\nabla f|^2\bigr)(\gamma_\theta(t))\,dt .
\]
Hence
\begin{equation}\label{eq:J-sharp}
    J(r,\theta)
    \le
    e^{\Phi(r,\theta)}r^{n-1}
    \qquad
    \text{for }0<r<c(\theta).
\end{equation}

Using the polar integration formula over the segment domain, as in the proof of
Theorem \ref{hhg}, \eqref{eq:J-sharp} gives
\[
\begin{aligned}
    \operatorname{Vol}(B_o(R))
    &=
    \int_{\mathbb S^{n-1}}
    \int_0^{\min\{R,c(\theta)\}}
    J(s,\theta)\,ds\,d\theta                                      \\
    &\le
    \int_{\mathbb S^{n-1}}
    \int_0^{\min\{R,c(\theta)\}}
    e^{\Phi(s,\theta)}s^{n-1}\,ds\,d\theta                         \\
    &\le
    \int_{\mathbb S^{n-1}}
    \int_0^R
    e^{\Phi(s,\theta)}s^{n-1}\,ds\,d\theta .
\end{aligned}
\]
This proves \eqref{eq:sharp-comparison-almost}.

Since
\[
    A_f(R)=\sup_{B_o(R)}(|\nabla f|^2-f),
\]
we have, for \(0\le s\le R\),
\[
    f-|\nabla f|^2\ge -A_f(R).
\]
Thus
\[
    \Phi(s,\theta)\le f(o)+A_f(R).
\]
Substituting this into \eqref{eq:sharp-comparison-almost}, we obtain
\[
\begin{aligned}
    \operatorname{Vol}(B_o(R))
    &\le
    e^{f(o)+A_f(R)}
    \int_{\mathbb S^{n-1}}\int_0^R s^{n-1}\,ds\,d\theta             \\
    &=
    e^{f(o)+A_f(R)}\omega_nR^n .
\end{aligned}
\]
This proves \eqref{eq:polynomial-growth-almost}.

If, for a fixed \(R>0\),
\[
    f(o)\le -A_f(R),
\]
then
\[
    \Phi(s,\theta)\le0
    \qquad
    \text{for }0<s\le R.
\]
Therefore \(e^{\Phi(s,\theta)}\le1\), and
\[
    \operatorname{Vol}(B_o(R))\le \omega_nR^n.
\]
This proves \eqref{eq:euclidean-volume-bound}.

We now prove the rigidity statement. Suppose equality in
\eqref{eq:sharp-comparison-almost} holds for every \(R>0\). Set
\[
    \Psi(r,\theta):=e^{\Phi(r,\theta)}r^{n-1}.
\]
Since \(\Psi(r,\theta)>0\) for all \(r>0\), the same segment-domain argument as
in the proof of Theorem \ref{hhg} gives
\[
    c(\theta)=\infty
    \qquad
    \text{for every }\theta\in\mathbb S^{n-1}.
\]
Hence
\[
    \operatorname{Cut}(o)=\varnothing.
\]
Moreover, equality of the volume integrals and the pointwise estimate
\eqref{eq:J-sharp} imply
\begin{equation}\label{eq:J-Psi}
    J(r,\theta)=e^{\Phi(r,\theta)}r^{n-1}
    \qquad
    \text{for all }r>0,\ \theta\in\mathbb S^{n-1}.
\end{equation}
Thus
\begin{equation}\label{eq:L-equals-Phi}
    L(r,\theta)=\Phi(r,\theta).
\end{equation}

Define
\[
    I(r,\theta)
    :=
    \int_0^r
    \frac1{s^2}
    \int_0^s
    t^2\bigl(f''(t,\theta)-\rho(t,\theta)\bigr)\,dt\,ds .
\]
By the \(p=r^2\) case of Theorem \ref{thm1.1}, we have
\begin{equation}\label{eq:L-leq-I}
    L(r,\theta)\le I(r,\theta).
\end{equation}

We now compare \(I\) with \(\Phi\). A direct calculation gives
\[
    \bigl(rI(r,\theta)\bigr)'
    =
    \int_0^r
    t\bigl(f''(t,\theta)-\rho(t,\theta)\bigr)\,dt
    =
    rf'(r)-f(r)+f(o)-\int_0^r t\rho(t,\theta)\,dt.
\]
On the other hand,
\[
    \bigl(r\Phi(r,\theta)\bigr)'
    =
    f(o)-\bigl(f-|\nabla f|^2\bigr)(r,\theta).
\]
Hence
\[
\begin{aligned}
    \bigl(r(\Phi-I)\bigr)'
    &=
    |\nabla f|^2-rf'
    +
    \int_0^r t\rho(t,\theta)\,dt                                    \\
    &=
    \bigl(|\nabla f|^2-(f')^2\bigr)
    +
    \left(f'-\frac r2\right)^2
    +
    \int_0^r
    t\left(\rho(t,\theta)-\frac12\right)dt .
\end{aligned}
\]
Since
\[
    |\nabla f|^2-(f')^2\ge0
    \qquad\text{and}\qquad
    \rho\ge\frac12,
\]
we have
\[
    \bigl(r(\Phi-I)\bigr)'\ge0.
\]
Moreover,
\[
    \lim_{r\to0^+}r(\Phi-I)=0.
\]
Therefore
\begin{equation}\label{eq:I-leq-Phi}
    I(r,\theta)\le \Phi(r,\theta).
\end{equation}

Combining \eqref{eq:L-equals-Phi}, \eqref{eq:L-leq-I}, and
\eqref{eq:I-leq-Phi}, we obtain
\[
    L(r,\theta)=I(r,\theta)=\Phi(r,\theta)
\]
for all \(r>0\) and all \(\theta\in\mathbb S^{n-1}\).

Since equality holds in \eqref{eq:I-leq-Phi}, the nonnegative terms in
\[
\bigl(r(\Phi-I)\bigr)'
=
\bigl(|\nabla f|^2-(f')^2\bigr)
+
\left(f'-\frac r2\right)^2
+
\int_0^r
t\left(\rho(t,\theta)-\frac12\right)dt
\]
must vanish identically. Hence
\[
    f'(r)=\frac r2,
    \qquad
    |\nabla f|^2=(f')^2,
\]
and
\[
    \int_0^r
    t\left(\rho(t,\theta)-\frac12\right)dt=0
    \qquad
    \text{for all }r>0.
\]
Since \(\rho\ge\frac12\), the last identity implies
\[
    \rho\equiv\frac12.
\]
Moreover,
\[
    \nabla f=\frac r2\partial_r,
\]
and integrating along each radial geodesic gives
\[
    f(r,\theta)=f(o)+\frac{r^2}{4}.
\]

We now verify the hypotheses of Theorem \ref{hhg} for the model function
\[
    p(r,\theta)=r^2.
\]
For this choice,
\[
    \mathcal K_p=0.
\]
Since
\[
    f''=\frac12
    \qquad\text{and}\qquad
    \rho\equiv\frac12,
\]
we have
\[
    f''-\rho=0=\mathcal K_p.
\]
Thus \(p=r^2\) satisfies the generalized Riccati equation \eqref{ricatii}.

Furthermore, from
\[
    f(r,\theta)=f(o)+\frac{r^2}{4}
    \qquad\text{and}\qquad
    \nabla f=\frac r2\partial_r,
\]
we get
\[
    f-|\nabla f|^2=f(o).
\]
Hence
\[
    \Phi(r,\theta)
    =
    f(o)-\frac1r\int_0^r f(o)\,dt
    =
    0.
\]
Therefore \eqref{eq:J-Psi} gives
\[
    J(r,\theta)=r^{n-1}.
\]
Consequently,
\[
    \operatorname{Vol}(B_o(R))
    =
    \int_{\mathbb S^{n-1}}\int_0^R s^{n-1}\,ds\,d\theta
    =
    \omega_nR^n
    =
    \int_{\mathbb S^{n-1}}\int_0^R
    p^{\frac{n-1}{2}}(s,\theta)\,ds\,d\theta
\]
for every \(R>0\). Since the polar model
\[
    dr^2+r^2g_{\mathbb S^{n-1}}
\]
is globally defined and smooth at the pole, all hypotheses of Theorem
\ref{hhg} are satisfied. Hence \((M^n,g)\) is globally isometric to
\(\mathbb R^n\).

Together with
\[
    f=f(o)+\frac{r^2}{4},
    \qquad
    \rho\equiv\frac12,
\]
this gives the Gaussian shrinking Ricci soliton, up to adding the constant
\(f(o)\) to the potential.
\end{proof}

Using the potential estimates above, we obtain the following corollary.
\begin{corollary}\label{cor4.7}
Assume the hypotheses of Theorem \ref{thm4.6}. After a constant rescaling of
the metric if necessary, suppose that
\[
    \rho\ge \frac12
\]
and that there exists a constant \(C_\rho>0\) such that
\[
    \left|\rho(x)-\frac12\right|
    \le
    \frac{C_\rho}{1+r(x)^2}
\]
for all \(x\in M\). Then there exists a constant \(C>0\), independent of \(R\),
such that
\[
    \operatorname{Vol}(B_o(R))\le CR^n
\]
for all \(R>0\).
\end{corollary}

\begin{proof}
By Theorem \ref{thm1.3}, it is enough to prove that the quantity
\[
    A_f(R)
    =
    \sup_{B_o(R)}\bigl(|\nabla f|^2-f\bigr)
\]
is bounded independently of \(R\).

From Corollary \ref{cor4.3} and the proof of Theorem \ref{thm4.5}, we have
\[
    |\nabla f|^2
    \le
    2\rho f+C_0
\]
for some constant \(C_0\). Hence
\[
    |\nabla f|^2-f
    \le
    (2\rho-1)f+C_0.
\]
By Theorem \ref{thm4.6}, after enlarging the constant if necessary, there
exists \(C_1>0\) such that
\[
    |f(x)|\le C_1(1+r(x)^2)
\]
for all \(x\in M\). On the other hand, the decay assumption gives
\[
    |2\rho(x)-1|
    \le
    \frac{2C_\rho}{1+r(x)^2}.
\]
Therefore
\[
    (2\rho-1)f
    \le
    |2\rho-1|\,|f|
    \le
    2C_\rho C_1.
\]
It follows that
\[
    |\nabla f|^2-f
    \le
    2C_\rho C_1+C_0
\]
on \(M\). Thus there exists a constant \(C_A>0\), independent of \(R\), such
that
\[
    A_f(R)\le C_A
    \qquad
    \text{for all }R>0.
\]
Applying Theorem \ref{thm1.3}, we obtain
\[
    \operatorname{Vol}(B_o(R))
    \le
    e^{f(o)+C_A}\omega_n R^n
    \qquad
    \text{for all }R>0.
\]
This proves the claim.
\end{proof}

\section*{Acknowledgements}

The author would like to thank his advisor Prof. Meng Zhu for inspiring
discussions and invaluable suggestions. The author is also grateful to the
reviewer for the careful reading of the manuscript and for the valuable comments
and suggestions, which helped improve the paper.

 \vspace{3em} % 增加 1em 的垂直间距 

\section*{Declaration of generative AI and AI-assisted technologies in the manuscript preparation process}

During the preparation of this work, the author used OpenAI's ChatGPT to assist
with language editing, improving readability, and checking the clarity and
internal consistency of parts of the mathematical exposition. After using this
tool, the author reviewed and edited the content as needed, independently
verified all mathematical statements and proofs, and takes full responsibility
for the content of the published article.

\end{sloppypar}

\begin{thebibliography}{99}
\bibitem{AE} A. Barros, E. Ribeiro Jr, Some characterizations for compact almost Ricci solitons. \emph{Proc. Amer. Math. Soc.}, 140, 1033-1040. (2012).

\bibitem{ABR} A. Barros, R. Batista, E. Ribeiro Jr, Compact almost Ricci solitons with constant scalar curvature are gradient. \emph{Monatsh. Math.}, 174, 29–39. (2014)


\bibitem{AzamiHajiaghasi2022}
S. Azami and S. Hajiaghasi, New volume comparison with almost Ricci soliton,
\textit{Commun. Korean Math. Soc.}, \textbf{37} (2022), no.~3, 839--849.



\bibitem{BakryEmery1985} D. Bakry, M. \'Emery, Diffusions hypercontractives, \textit{Séminaire de probabilités, XIX}, 1983/84, pp. 177–206, Lecture Notes in Mathematics, vol. 1123, Springer, Berlin, 1985.

\bibitem{BakryQian} D. Bakry, Z.-M. Qian, Some new results on eigenvectors via dimension, diameter and Ricci curvature, \textit{Advances in Mathematics}, vol. 155, pp. 98–153, 2000.

\bibitem{CCD} G. Catino, L. Cremaschi, Z. Djadli, C. Mantegazza, L. Mazzieri, The Ricci-Bourguignon flow. \emph{Pacific J. Math.}, 287(2), 337–370. (2017). 

\bibitem{CM} G. Catino, L. Mazzieri, Gradient Einstein solitons. \emph{Nonlinear Anal.}, 132, 66--94. (2016).

\bibitem{CZ} H.-D. Cao, D. Zhou, On complete gradient shrinking Ricci solitons. \emph{J. Differential Geom.}, 85(2), 175-186. (2010).

\bibitem{CCZ} H.-D. Cao, B.-L. Chen, X.-P. Zhu, Recent Developments on Hamilton’s Ricci flow. \emph{Surveys in Differential Geometry XII}. (2008).

\bibitem{CRZ} X. Cheng, E. Ribeiro Jr, D. Zhou, Volume growth estimates for Ricci solitons and quasi-Einstein manifolds. \emph{J. Geom. Anal.}, 32(2), 62. (2022).

\bibitem{FAJR} F. E. S. Feitosa, A. A. Freitas Filho, J. N. V. Gomes, R. S. Pina, Gradient Ricci almost soliton twisted product. \emph{J. Geom. Phys.}, 143, 22–32. (2019).

\bibitem{GrayVanhecke1979}
A. Gray and L. Vanhecke, Riemannian geometry as determined by the volumes of small geodesic balls,
\textit{Acta Math.}, \textbf{142} (1979), 157--198.

\bibitem{Ha1}R. S. Hamilton, The formation of singularities in the Ricci flow. \emph{Surveys in Differential Geometry (Cambridge, MA, 1993)}, 2, 7-136. (1995).

\bibitem{Ha3}R. S. Hamilton, The Ricci flow on surfaces. \emph{Contemp. Math.}, 71, 237-261. (1988).

\bibitem{Li}P. Li, Geometric analysis. Cambridge Studies in Advanced Mathematics, 134. Cambridge University Press. (2012).

\bibitem{MW} O. Munteanu and J. Wang, Geometry of manifolds with densities, \emph{Advances in Math}. 259 (2014), 269–305.

\bibitem{N}A. Naber, Noncompact shrinking four solitons with nonnegative curvature. \emph{J. Reine Angew. Math.}, 645, 125-153. (2010).

\bibitem{NW}L. Ni, N. Wallach, On a classification of gradient shrinking solitons. \emph{Math. Res. Lett.}, 15(5), 941-955. (2008).

\bibitem{Pel1} G. Perelman, The entropy formula for the Ricci flow and its geometric applications. arXiv:math.DG/0211159. (2002).

\bibitem{Pel2} G. Perelman, Ricci flow with surgery on three manifolds. arXiv:math.DG/0303109. (2003).

\bibitem{Pe} P. Petersen, Riemannian geometry. New York: Springer. (2006).

\bibitem{Pw1} P. Petersen, W. Wylie, On the classification of gradient Ricci solitons. \emph{Geom. Topol.}, 14, 2277–2300. (2010).

\bibitem{Pw2} P. Petersen, W. Wylie, Rigidity of gradient Ricci solitons. \emph{Pacific J. Math.}, 241, 329–345. (2009).

\bibitem{PRR} S. Pigola, M. Rigoli, M. Rimoldi, A. G. Setti, Ricci almost solitons. \emph{Ann. Sci. Norm. Super. Pisa-Cl. Sci.}, 10(4), 757-799. (2011).

\bibitem{WW} G. Wei, W. Wylie, Comparison geometry for the Bakry-Emery Ricci tensor. \emph{J. Differential Geom.}, 83, 377–405. (2009).

\bibitem{Z} Z. H. Zhang, On the completeness of gradient Ricci solitons. \emph{Proc. Amer. Math. Soc.}, 137, 2755–2759. (2009).

\bibitem{ZZ} Q. S. Zhang, M. Zhu, New volume comparison results and applications to degeneration of Riemannian metrics. \emph{Adv. Math.}, 352, 1096-1154. (2019).
\end{thebibliography}
\end{document}